\documentclass[10.5pt,a4paper]{article}
\usepackage[latin1]{inputenc}
\usepackage{amsmath}
\usepackage{amsfonts}
\usepackage{amssymb} 

\usepackage{color}

\usepackage{authblk}

\usepackage{textcomp}
\usepackage{makeidx}
\usepackage{graphicx}
\usepackage[bottom]{footmisc}
\usepackage{amsthm}
\usepackage{csquotes}
\usepackage{hyperref}
\hypersetup{colorlinks=true,linkcolor=blue,citecolor=blue}
\numberwithin{equation}{section}

\author{Giuseppe Scola\\
	Gran Sasso Science Institute (GSSI), \\ Viale Francesco Crispi, 7, 67100, L'Aquila (IT).
	\\
	giuseppe.scola@gssi.it}
\newtheorem{theorem}{Theorem}[section]
\newtheorem{corollary}{Corollary}[theorem]
\newtheorem{lemma}[theorem]{Lemma}
\newtheorem{remark}{Remark}[section]
\title{Local moderate and precise large deviations via cluster expansions}

\newcommand{\probBC}{\mathbb{P}^{\mathbf{\zero}}_{\Lambda,\mu_0}}
\newcommand{\e}{\boldsymbol{\eta}}

\newcommand{\canBC}{Z^{\mathbf{0}}_{\Lambda,\beta}}

\newcommand{\can}{Z_{\Lambda,\beta}} 

\newcommand{\GcanBC}{\Xi^{\mathbf{0}}_{\Lambda,\beta}} 

\newcommand{\Gcan}{\Xi_{\Lambda,\beta}} 

\newcommand{\PNL}{P_{N,|\Lambda|}} 

\newcommand{\Bbl}{B_{\Lambda,\beta}} 

\newcommand{\R}{\mathbb{R}} 

\newcommand{\tN}{\tilde{N}} 
\newcommand{\bN}{\bar{N}_{\Lambda}} 
\newcommand{\bRL}{\bar{\rho}_{\Lambda}} 
\newcommand{\tRL}{\tilde{\rho}_{\Lambda}} 

\newcommand{\F}{\mathcal{F}_{\Lambda,\beta,\mathbf{0}}}
\newcommand{\fgc}{f^{GC}_{\Lambda,\beta,\mathbf{0}}}

\newcommand{\f}{\mathcal{F}}
\newcommand{\p}{\mathcal{P}}

\newcommand{\s}{\lesssim}

\newcommand{\zero}{\mathbf{0}}

\makeatletter
\def\@makefnmark{}
\makeatother
\newcommand{\myfootnote}[2]{\footnote{{#1} #2}}

\date{}

\begin{document}
	
	\maketitle
	\begin{abstract}
		We consider a system of classical particles confined in a box $\Lambda\subset\mathbb{R}^d$ with 
		zero boundary conditions  interacting via a stable and regular pair potential. 
		Based on the validity of the cluster expansion for the canonical partition function in the high temperature - low density regime we prove
		moderate and precise large deviations from the mean value of the number of particles with
		respect to the grand-canonical Gibbs measure. In this way we have a direct method of computing both the exponential rate as well as the pre-factor and obtain explicit error terms. 
		Estimates comparing with the infinite volume versions of the above are also provided.
	\end{abstract}
	
	\tableofcontents
	
	\myfootnote{{\it{2010 Mathematics Subject Classification.}} 60F05, 60F10, 82B05.
		
		
		\quad cluster expansion.}

\section{Introduction}
Limit theorems in probability have been very useful in expressing thermodynamic quantities in statistical mechanics in terms of variational principles \cite{lanford1973entropy}. On the other hand, other more explicit methods have been developed in the mathematical physics literature for the calculation of the thermodynamic potentials such as cluster expansions which represent them as absolutely convergent series \cite{mayerstatistical},\cite{poghosyan2009abstract}. Comparing the two, one could think that the latter might be a way for an explicit calculation of the expressions appearing in the probabilistic limit theorems. In fact, the aim of this paper is to give a precise example in a general framework where cluster expansions could give explicit forms for the functionals appearing in the limit theorems. However, the price to pay is that we have to restrict ourselves to the rather small range of validity of the cluster expansions, while variational principles cover even phase transitions. We hope however that this could be enlightening for the better understanding of these methods, but also useful when higher order corrections and/or pre-factors to the asymptotic formula of large deviations is needed. Furthermore, given the validity of
the cluster expansion as well as of the virial inversion for inhomogeneous density functionals \cite{jansen2019virial} the connection
outlined above	can be extended in a broader set of models.



More precisely, we work in the context of the classical gas in a box with zero boundary conditions. The particles interact via a stable and regular potential. Large deviations for such systems have been developed by Georgii \cite{georgii1994large} in terms of point processes and the question of equivalence of ensembles has been addressed. More recently, the fluctuations have also been studied \cite{cancrini2017ensemble} together with the equivalence of the canonical and microcanonical ensemble. 
In a similar spirit but for the Ising model, in \cite{del1974local} the author performs a central limit theorem expansion using the characteristic function and Gnedenko's method.
An instructive review on the topic of moderate and precise large deviations for the Ising model is given by Dobrushin and Shlosman in \cite{dobrushin1994large} with a rich bibliographical account. 
In this work, the authors start from a probabilistic large deviation approach
and obtain precise large deviations as well as moderate deviations using the characteristic function.
Furthermore, they focus on the more interesting phase transition regime.

Our work is in the spirit of the aforementioned references and adds some more information in the following directions:
\begin{enumerate}
	\item Given the validity of the cluster expansion in the canonical ensemble (in the high temperature - low density regime) we present a direct calculation of the exponential rate as well as of the pre-factor (via absolutely convergent power series) appearing in the limit theorems avoiding the
	computations based on the characteristic function as in \cite{cancrini2017ensemble}, \cite{del1974local}, and \cite{dobrushin1994large}. Furthermore, comparisons between the finite and the infinite volume functionals are provided based
	on related results - \cite{pulvirenti2015finite} - whenever cluster expansion holds.
	\item We consider zero boundary conditions and see how they influence the choice of the finite volume functionals, as opposed to the simpler case of periodic boundary conditions.  	
	\item We deal with canonical/grand-canonical equivalence rather than the microcanonical as in \cite{cancrini2017ensemble}. 
\end{enumerate}


The paper is organized as follows: We start (Section~\ref{S1}) fixing the notation and recalling the cluster expansion for the canonical partition function as it is presented in \cite{pulvirenti2012cluster}. Then, we present the main results (Section \ref{SubThs}). In Section \ref{Sub1} we give a brief comparison between our approach and the one presented in \cite{del1974local} and \cite{dobrushin1994large}.  In Section \ref{ProofSec} we give the backbone of the proofs, while all technical details have been incorporated in a series of lemmas which are given and proved in Section~\ref{S2}.


\section{Formal description of the model: notation, main definitions and useful results} 
\label{S1}

We consider a system of $N$ indistinguishable particles at inverse temperature $\beta$, described by a configuration 
$\mathbf{q}=\{q_1,...,q_N\}$ (where $q_i$ is the position of the $i^{th}$ particle), confined  in a box $\Lambda:=\big(\frac{l}{2},\frac{l}{2}\big]^d\subset\R^d$ (for some $l>0$).  The particles interact with a pair potential $V:\R^d\rightarrow\R\cup\{\infty\}$ which is an even function and which satisfies the following assumptions:
	
	\textbf{Assumption 1:} there exists a constant $B\ge0$ such that for all $n\in\mathbb N$ and for every configuration $(x_1,..,x_n)\in(\R^d)^n$

	\begin{equation}
	\sum_{1\le i<j\le n}V(x_i-x_j)\ge-Bn\;\;\;\mathrm{(stability)};
	\label{assumption1}
	\end{equation}
	
	\textbf{Assumption 2:} 
	\begin{equation}
	C(\beta):=\int_{\R^d}|e^{-\beta V(x)}-1|dx<\infty\;\;\;\mathrm{(regularity)},
	\label{assumption2}
	\end{equation}
	for all $\beta>0$. 
	
	We assume that the particles in $\Lambda$ do not interact with the particles in $\Lambda^c:=\mathbb{R}^d\setminus\Lambda$ (zero boundary conditions), hence, the Hamiltonian is given by
	
	\begin{equation}
	H^{\zero}_{\Lambda}(\mathbf{q}):=\sum_{1\le i<j\le N}V(q_i-q_j)
	\label{Hamiltonian0}
	\end{equation}
	where $\mathbf{q}\in\Lambda^N$.

The \textit{canonical} and the \textit{grand-canonical partition function} are given by

\begin{equation}	
\canBC(N):=\frac{1}{N!}\int_{\Lambda^N}dq_1...dq_N\;e^{-\beta H^{\zero}_{\Lambda}(\mathbf{q})}
\label{Can0}
\end{equation}
and
\begin{equation}
\GcanBC(\mu):=\sum_{N\ge 0}e^{\beta\mu N} \canBC(N),
\label{GrandCan0}
\end{equation}
where $\mu\in\R$ is the chemical potential. When  we do not need to specify the conditions at the boundary we write $\can(N)$ and $\Gcan(\mu)$. 

We define the \textit{finite\;volume\;free\;energy} and the \textit{finite\;volume\;pressure} as:

\begin{equation}
f_{\Lambda,\beta,\zero}(N):=-\frac{1}{\beta|\Lambda|}\log Z^{\zero}_{\Lambda,\beta}(N),
\label{FreeE}
\end{equation}
and
\begin{equation}
p_{\Lambda,\beta,\zero}(\mu):=\frac{1}{\beta|\Lambda|}\log\Xi^{\zero}_{\Lambda,\beta}(\mu).
\label{FVP}
\end{equation}

As before when we do not need to specify the dependence of the previous quantities on the boundary conditions we will use the notation $f_{\Lambda,\beta}(\cdot)$ and 
$p_{\Lambda,\beta}(\cdot)$.  For later use we also introduce the {\it{grand-canonical free energy}}:
	\begin{equation}
	\beta f_{\Lambda,\beta,\zero}^{GC}(\rho):=\sup_{\mu\in\mathbb{R}}\left\{\beta\mu\rho-\beta p_{\Lambda,\beta,\zero}(\mu) \right\}
	\label{GCFE1}
	\end{equation}
	and the {\it{canonical pressure}}:
	\begin{equation}
	\beta p^C_{\Lambda,\beta,\zero}(\mu):=\sup_{N\in\mathbb{N}}\left\{\beta \frac{N}{|\Lambda|}\mu-\beta f_{\Lambda,\beta,\zero}(N)\right\},
	\label{N*}
	\end{equation}
	for a given $\rho\in(0,1)$ and $\mu\in\mathbb{R}$.

We define the (infinite volume) \textit{free\;energy} and \textit{pressure} as

\begin{equation}  f_{\beta}(\rho):=\lim_{\substack{\Lambda\rightarrow\R^d\\N/|\Lambda|\rightarrow\rho}}f_{\Lambda,\beta}(N)
\label{Infinit}
\end{equation}
and
\begin{equation}
p_{\beta}(\mu):=\lim_{\Lambda\rightarrow\R^d}p_{\Lambda,\beta}(\mu);
\end{equation}
they are related via the Legendre transform:

\begin{equation}
\beta f_{\beta}(\rho)=\sup_{\mu}\left\{\beta \rho\mu- \beta p_{\beta}(\mu)\right\}
\label{FreeELT}
\end{equation}
and
\begin{equation}
\beta p_{\beta}(\mu)=\sup_{\rho}\left\{\beta \rho\mu-\beta f_{\beta}(\rho)\right\}.
\label{PressLT}
\end{equation}
Let us note that given some $\mu_0$, evaluating at the supremum we obtain:
\begin{equation}
\mu_0=f'_{\beta}(\rho_0)\Leftrightarrow \rho_0=p_{\beta}'(\mu_0),
\label{MU}
\end{equation}
(if $f_\beta$ is strictly convex) 
which gives 
\begin{equation}
\beta f_{\beta}(\rho_0)=\beta\rho_{0}\mu_0-\beta p_{\beta}(\mu_0).
\end{equation}

Given a chemical potential $\mu_0$, the grand-canonical probability measure with zero boundary conditions $\mathbb{P}^{\zero}_{\Lambda,\mu_0}(\cdot)$ is given by 

\begin{equation}
\mathbb{P}^{\zero}_{\Lambda,\mu_0}(\mathrm{d}\mathbf{q}):=\bigotimes_{N\geq 0}\frac{e^{\beta\mu_0 N}e^{-\beta H_{\Lambda}^{\zero}(\mathbf{q})}dq_1\cdot\cdot\cdot dq_N}{\GcanBC(\mu_0)\;N!},
\label{GcProbE}
\end{equation} 
where by $dq_i$ we denote the Lebesgue measure on $\Lambda\subset\mathbb R^d$.
The mean value of the particle density and the variance calculated using the grand-canonical probability measure are denoted by: 

\begin{equation}
\bRL:=\frac{1}{|\Lambda|}\mathbb{E}^{\zero}_{\Lambda,\mu_0}\left[N\right]=\frac{\partial}{\partial\mu} p_{\Lambda,\beta,\zero}(\mu)\bigg|_{\mu=\mu_0},\;\;\bN:=\lfloor\bRL|\Lambda|\rfloor
\label{MeanValue}
\end{equation}
and
\begin{equation}
\sigma^2_{\Lambda,\zero}(\mu_0):=\mathbb{E}^{\zero}_{\Lambda,\mu_0}\left[\frac{(N-\bRL|\Lambda|)^2}{|\Lambda|}\right]=\frac{1}{\beta}\frac{\partial^2}{\partial\mu^2} p_{\Lambda,\beta,\zero}(\mu)\bigg|_{\mu=\mu_0}.
\label{GeneralVariance}
\end{equation}

We define the deviation of order $\alpha\in [1/2,1]$ from $\bN$ as follows:
\begin{equation}
\tN\equiv\tN(u,\alpha):=\bN+ u|\Lambda|^{\alpha},
\label{GeneralDeviations}
\end{equation} 
with $u\in\R$ (in such a way that $\tN\in\mathbb{N}$). Furthermore, we also denote with $A_{\tN}$, the set of particle configurations of $\tN$ particles inside $\Lambda$:
	\begin{equation}\label{SetA}
	A_{\tN}:=\{\mathbf{q}\equiv\{q_i\}_{i \geq 1}, q_i\in\R^d |\, |\mathbf{q}\cap\Lambda|=\tN\}.
	\end{equation}

	In this paper we study the (grand-canonical) probability  - \eqref{GcProbE} - of the number of particles defined in \eqref{GeneralDeviations}, following a different approach than usually.	
	As we anticipated in the introduction and as it will be explained better in next sections, 
	our method is based on the validity of the cluster expansion for the canonical partition function. This approach is also beneficial whenever uniform estimates in the volume are needed (in order to pass to the limit). On the one hand, this is a more direct and explicit method for the calculation of the deviations, but on the other, it is quite restrictive as it is valid only for small values of the density.
	Therefore, following \cite{pulvirenti2012cluster}, we define for all $n\ge1$ 
	\begin{equation}
	F_{\beta,N,\Lambda}(n):=\frac{1}{n+1}\PNL(n) B_{\Lambda,\beta}(n),
	\label{Cluster-Coefficient}
	\end{equation}
	where, for the definition of $B_{\Lambda,\beta}(n)$ (not needed here) we refer to   \eqref{B} and with
	\begin{equation}
	\PNL(n):=\begin{cases}\frac{(N-1)\cdot\cdot\cdot(N-n)}{|\Lambda|^n}\;\;\;\mathrm{if}\;n<N,\\
	\\
	0\;\;\;\mathrm{otherwise}.
	\end{cases}
	\label{PNL}
	\end{equation}
	Hence, having these quantities, the cluster expansion of the canonical partition function allows us to rewrite the logarithm of \eqref{Can0} as 
	\begin{equation}
	\frac{1}{|\Lambda|}\log \canBC(N)=\frac{1}{|\Lambda|}\log\frac{|\Lambda|^N}{N!}+\frac{N}{|\Lambda|}\sum_{n\ge 1}F_{\beta,N,\Lambda}(n),
	\label{CanCE1}
	\end{equation}
	assuming that:	\newline
	
	\textbf{Condition ($\star$):}
	\begin{equation}
	\frac{N}{|\Lambda|} C(\beta)<c_0,
	\label{AssumptioCCE}
	\end{equation}
	with $c_0\equiv c_0(\beta,B)\in\mathbb{R}^+$ an explicit constant, see \cite{pulvirenti2012cluster}. Let us note that Condition $(\star)$ is a low density - high temperature condition, where the potential considered satisfies Assumptions 1 and 2.
	
	Furthermore, we know that there exist constants $C,c>0$ such that  for every $N$ and $\Lambda$ the coefficients $F_{\beta,N,\Lambda}(n)$ satisfy 
	\begin{equation}
	\left|F_{\beta,N,\Lambda}(n)\right|\le C e^{-cn},\;\;\mathrm{for\;all}\;n\ge1.
	\label{absCan}
	\end{equation}
	
	The last quantity we need to introduce before stating the main results is the free energy $\F$ which is a function of the density $\rho\in(0,1)$, defined as:  
	\begin{equation}
	\mathcal{F}_{\Lambda,\beta,\zero}(\rho):=\frac{1}{\beta}\left\{\rho(\log\rho-1)-\sum_{n\ge1}\frac{1}{n+1}\mathcal{P}_{n+1}(\rho)B_{\Lambda,\beta}(n)\right\}.
	\label{FreeERL1}
	\end{equation}
	Here, $\mathcal{P}_{n+1}(\rho)$ is a polynomial of degree $n+1$ evaluated at $\rho$ given by
	\begin{equation}
	\mathcal{P}_{n+1}(\rho):=\begin{cases}\rho\left(\rho-\frac{1}{|\Lambda|}\right)\cdot\cdot\cdot\left(\rho-\frac{n}{|\Lambda|}\right)\;\;\;\;\mathrm{if}\;\frac{n}{|\Lambda|}<\rho,\\
	\\
	0\;\;\;\;\mathrm{otherwise},
	\end{cases}
	\label{PNL1}
	\end{equation}
	and $B_{\Lambda,\beta}(n),\;n\ge1$ are the same coefficients presented in \eqref{Cluster-Coefficient}.
	For all $N\in\mathbb{N}$ and $\rho_{\Lambda}:=N/|\Lambda|$, from \eqref{PNL} and \eqref{PNL1} we have 
	\begin{equation}
	\mathcal{P}_{n+1}(\rho_{\Lambda})=\rho_{\Lambda}P_{\rho_{\Lambda}|\Lambda|,|\Lambda|}(n).
	\end{equation}	
	Note that this new free energy is a version of  \eqref{FreeE} expressed using   \eqref{CanCE1}, which satisfies
	\begin{equation}
	| f_{\Lambda,\beta,\zero}(N)-\F(\rho_{\Lambda})|=|S_{|\Lambda|}(\rho_{\Lambda})|\s\frac{\log\sqrt{|\Lambda|}}{|\Lambda|}
	\label{A}
	\end{equation}	
	with $S_{|\Lambda|}(\rho_{\Lambda})$ given by \eqref{StirlingEq} and where the last inequality follows from \eqref{StirlingBound}.
	When we do not need to specify the dependence on the boundary conditions we will use the notation $\mathcal{F}_{\Lambda,\beta}(\cdot)$ and we will denote with  $\f^{(m)}_{\Lambda,\beta}(\cdot)$ and the {\it m-th} derivative of  $\f_{\Lambda,\beta}(\cdot)$.

\section{Main results}
\label{SubThs}
Now we can state the main results of this paper which will be proved in Sections \ref{ProofSec}.

\begin{theorem}[Precise Large Deviations]
	Let $\mu_0\in\R$ be a chemical potential and let $\tilde{N}$ be a fluctuation given by \eqref{GeneralDeviations} with $\alpha=1$ such that condition $(\star)$ holds. 
	Let also be $V:\R^d\rightarrow\R\cup\{\infty\}$ a pair potential which satisfies Assumptions 1 and 2 with zero boundary conditions outside a box $\Lambda\subset\mathbb{R}^d$.  
	
	Moreover, let $\tilde{\mu}_{\Lambda}\in\R$ be the chemical potential that corresponds to the supremum at  equation \eqref{GCFE1}, $\tRL:=\tN/|\Lambda|$ and $A_{\tN}$ as in \eqref{SetA}. We have: 
	\begin{equation}
	\left|\probBC\left(A_{\tilde{N}}\right)-\frac{e^{-|\Lambda|I^{GC}_{\Lambda,\beta,\zero}\left(\tRL;\bRL\right)}}{\sqrt{2\pi D_{\Lambda,\zero}(\tRL^*)|\Lambda|}}\right|\le 
	\frac{Ce^{-|\Lambda|I^{GC}_{\Lambda,\beta,\zero}\left(\tRL;\bRL\right)}}{|\Lambda|} 
	\label{P3}
	\end{equation}
	where 
	\begin{equation}
	I^{GC}_{\Lambda,\beta,\zero}\left(\tRL;\bRL\right):=\beta\left[\fgc(\tRL)-\fgc(\bRL)-\mu_0\left(\tRL-\bRL\right)\right],
	\label{LDO}
	\end{equation}
	and
	\begin{equation}
	D_{\Lambda,\zero}(\tRL^*):=\left[\beta \F''(\tRL^*)\right]^{-1}.
	\label{VAR}
	\end{equation}
	Here $\tRL^*=\tN^*/|\Lambda|$, with $\tN^*$  the number of particles where the supremum at equation \eqref{N*} occurs for $\mu=\tilde{\mu}_{\Lambda}$.
	\label{Th1}	
\end{theorem}

Next, for the moderate deviations, thanks to Lemma \ref{LemmaBN}, we can recenter the fluctuation $\tN$ given by \eqref{GeneralDeviations}  around the number of particles which satisfies \eqref{N*} when $\mu=\mu_0$. Hence, denoting with $N^*$ this number of particles, we can write
\begin{equation}
\tN=N^*+u'|\Lambda|^{\alpha}
\label{GeneralDeviations1}
\end{equation}
for some $\alpha\in[1/2,1)$ and $u'$ (depending on $u$).

\begin{theorem}[Local Moderate Deviations.]
	Let $\mu_0\in\R$ be a chemical potential and $N^*$ the number of particles where the supremum at equation \eqref{N*} occurs for $\mu=\mu_0$, such that condition $(\star)$ holds. 
	Let also be $V:\R^d\rightarrow\R\cup\{\infty\}$ a pair potential  which satisfies Assumptions 1 and 2 with zero boundary conditions  outside a box $\Lambda\subset\mathbb{R}^d$.
	
	For $\tilde N$ and the set $A_{\tilde N}$ respectively given by \eqref{GeneralDeviations1} and
	\eqref{SetA} with $\alpha\in[1/2,1)$ and denoting with $\rho^*_{\Lambda}:=N^*/|\Lambda|$, we have:
	\begin{equation}
	\left|\probBC(A_{\tN})-\frac{\exp\left\{-\frac{(u')^2|\Lambda|^{2\alpha-1}}{2D^{\alpha}_{\Lambda,\zero}(\rho^*_{\Lambda})}\right\}}{\sqrt{2\pi D^{\alpha,+}_{\Lambda,\zero}(\rho^*_{\Lambda})|\Lambda|}}\right|\le \frac{2e^{-\frac{(u')^2|\Lambda|^{2\alpha-1}}{2D^{\alpha}_{\Lambda,\zero}(\rho^*_{\Lambda})}}E_{|\Lambda|}(\alpha,u',\rho^*_{\Lambda})}{\sqrt{2\pi D^{\alpha,+}_{\Lambda,\zero}(\rho^*_{\Lambda})|\Lambda|}}
	\end{equation}
	where
	\begin{equation}
	D^{\alpha}_{\Lambda, \zero}(\rho^*_{\Lambda}):=\left[\beta\F''(\rho^*_{\Lambda})+\beta\sum_{m=3}^{m(\alpha)-1}\frac{2(u')^{m-2}\F^{(m)}(\rho^*_{\Lambda})}{m!|\Lambda|^{(m-2)(1-\alpha)}}\right]^{-1},
	\label{Var2}
	\end{equation}
	\begin{equation}
	D^{\alpha,+}_{\Lambda, \zero}(\rho^*_{\Lambda}):=\left[\beta\F''(\rho^*_{\Lambda})+\beta\sum_{m=3}^{m(\alpha)-1}\frac{2(u')^{m-2}|\F^{(m)}(\rho^*_{\Lambda})|}{m!|\Lambda|^{(m-2)(1-\alpha)}}\right]^{-1}.
	\label{Var1}
	\end{equation}
	Here, $m(\alpha)$ is given by \eqref{m-Coda} and  $E_{|\Lambda|}(\alpha,u',\rho^*_{\Lambda})$ is an error term of order $|\Lambda|^{-[(m(\alpha)(1-\alpha)-1]}$ defined via cluster expansion given by \eqref{Error}. 
	
	\label{Th2}
\end{theorem}

\begin{corollary}[Local Central Limit Theorem.]
	Under the same assumptions as in Theorem \ref{Th2} for $\alpha=1/2$ we have that
	\begin{equation}
	\left|\probBC(A_{\tN})-\frac{\exp\left\{-\frac{(u')^2}{2 D_{\Lambda,\zero}(\rho^*_{\Lambda})}\right\}}{\sqrt{2\pi D_{\Lambda,\zero}(\rho^*_{\Lambda})|\Lambda|}}\right|\le \frac{2 e^{-\frac{(u')^2}{2 D_{\Lambda,\zero}(\rho^*_{\Lambda})}}E_{|\Lambda|}(1/2,u',\rho^*_{\Lambda})}{\sqrt{2\pi D_{\Lambda,\zero}(\rho^*_{\Lambda})|\Lambda|}},
	\end{equation}
	where, using \eqref{VAR}, 
	\begin{equation}
	D_{\Lambda,\zero}(\rho^*_{\Lambda})=\left[\beta\F''(\rho^*_{\Lambda})\right]^{-1}
	\end{equation}
	and $E_{|\Lambda|}(1/2,u',\rho^*_{\Lambda})$ is an error term of order $|\Lambda|^{-1/2}$ defined via cluster expansion and given by \eqref{Error}. 
	\label{Corollary}
\end{corollary}

\section{General theory of large and moderate deviations vs our approach: a brief comparison}
\label{Sub1}
In this section we compare this work to the exiting approach. Recalling the general theory of large deviations (\cite{den2008large}, \cite{dobrushin1994large}, \cite{ellis2007entropy}, \cite{georgii1994large}),
for a fixed chemical potential $\mu_0$ we define the (finite volume) logarithmic generating function of the moments of \eqref{GcProbE}:
\begin{equation}
L^{\zero}_{\Lambda,\beta, \mu_0}(\mu):=\log\left[\sum_{N\ge0}\probBC(A_N)e^{\beta\mu N}\right],
\label{H}
\end{equation}
where the set $A_{N}$ is given by \eqref{SetA} for a $N\in\mathbb{N}$ instead of $\tN$.

	Let us note that from \eqref{FVP} and \eqref{GcProbE} we have 
	\begin{equation}
	L^{\zero}_{\Lambda,\beta,\mu_0}(\mu)=\beta|\Lambda|\left[p_{\beta,\Lambda,\zero}(\mu+\mu_0)-p_{\beta,\Lambda,\zero}(\mu_0)\right].
	\label{relation1}
	\end{equation}
	Furthermore,  from \eqref{MeanValue}, \eqref{GeneralVariance} and \eqref{H} we get the following equivalences 
	\begin{equation}
	\bRL|\Lambda|=\frac{1}{\beta}\frac{d}{d\mu}L^{\zero}_{\Lambda,\beta,\mu_0}(\mu)\bigg|_{\mu=0},
	\end{equation}
	\begin{equation}
	\sigma^{2}_{\Lambda,\zero}(\mu_0)|\Lambda|=\frac{1}{\beta^2}\frac{d^2}{d\mu^2}L^{\zero}_{\Lambda,\beta,\mu_0}(\mu)\bigg|_{\mu=0}.
	\end{equation}
	In general, we also denote by $G^{m}_{\Lambda,\beta,\zero}$ the $m$-th momentum $(m>2)$, which can be defined as:
	\begin{equation}
	G^{m}_{\Lambda,\zero}:=\frac{1}{\beta^{m}}\frac{d^m}{d\mu^m}L^{\zero}_{\Lambda,\beta,\mu_0}(\mu)\bigg|_{\mu=0}.
	\end{equation}
Let us define the characteristic function as
\begin{equation}\label{char}
\varphi_{\Lambda,\mu'}(t):=\sum_{N\ge0}\mathbb{P}_{\Lambda,\mu'}^{\zero}(A_{N})e^{itN},
\end{equation} 
where for $\mu'=\mu+\mu_0$,
\begin{equation}\label{char2}
\mathbb{P}_{\Lambda,\mu+\mu_0}^{\zero}(A_{N})=\exp\left\{-L^{\zero}_{\Lambda,\beta,\mu_0}(\mu)+\beta\mu N\right\}\probBC(A_{N})
\end{equation}
represents the ``excess (by $\mu$) probability measure".

	We will denote with $\sim$ the asymptotic behavior of two sequences, i.e.,  $a_n \sim b_n\iff\lim_{n\to\infty}\frac{a_n}{b_n}=1$. Hence, considering a deviation $\tN$ given by \eqref{GeneralDeviations} with $\alpha=1$, in order to compute its probability one can use the excess measure optimizing over $\mu$
	such that $\mathbb{P}_{\Lambda,\mu+\mu_0}^{\zero}(A_{\tilde N})\sim 1$, i.e., by making it ``central"
	with respect to the new measure. In this way we obtain that asymptotically as $\Lambda\to\mathbb R^d$:
\begin{equation*}
\mathbb{P}_{\Lambda,\mu_0}^{\zero}(A_{\tilde N})
\sim \exp\left\{- \mathcal{I}^{\zero}_{\Lambda,\beta,\mu_0}(\tN)\right\}\end{equation*}
where 
\begin{equation}
\mathcal{I}^{\zero}_{\Lambda,\beta,\mu_0}(\tN):=\sup_{\mu}\left\{\beta\mu \tN-L^{\zero}_{\Lambda,\beta,\mu_0}(\mu)  \right\}.
\label{I}
\end{equation}

However, if one needs a more precise formula 
one way is by inverting \eqref{char}:
\begin{equation}\label{invert}
\mathbb{P}_{\Lambda,\tilde\mu_{\Lambda}}^{\zero}(A_{\tilde N})=
\frac{1}{2\pi}\int_{-\pi}^{\pi}e^{-it\tN}\varphi_{\Lambda,\tilde\mu_{\Lambda}}(t)dt,
\end{equation} 
where by $\tilde\mu_{\Lambda}$ we denote the optimal chemical potential found in \eqref{I} and where
\begin{equation}
\frac{1}{2\pi}\int_{-\pi}^{\pi}e^{-it\tN}\varphi_{\Lambda,\tilde\mu_{\Lambda}}(t)dt
= \left(\sqrt{2\pi\sigma^2_{\Lambda,\e}(\tilde{\mu}_{\Lambda})|\Lambda|}\right)^{-1}(1+\ldots)\nonumber
\label{InvCarFun}
\end{equation}

 For more details for the above equality we refer to \cite{den2008large} sections I.3, I.4 and \cite{dobrushin1994large} section 2.1, formulas (2.1.15)-(2.1.20) and (2.1.31)-(2.1.34). In this paper we provide an alternative and more direct way for computing this pre-factor. 

For later use note that from \eqref{relation1}, by normalizing \eqref{I} and letting $\rho_{\Lambda}=N/|\Lambda|$ we have:

	\begin{eqnarray}
	I^{\zero}_{\Lambda,\beta,\mu_0}(\tRL) & := & \frac{1}{|\Lambda|}\mathcal{I}^{\zero}_{\Lambda,\beta,\mu_0}(\tN)=\nonumber\\
	& = & \sup_{\mu}\left\{\beta\mu\tRL -\beta p_{\Lambda,\beta,\zero}(\mu+\mu_0)+\beta p_{\Lambda,\beta,\zero}(\mu_0)\right\}\nonumber\\
	& = & \sup_{\mu}\left\{\beta\mu\tRL +\beta\mu_0\tRL- \beta\mu_0\tRL-\beta p_{\Lambda,\beta,\zero}(\mu+\mu_0)+\beta\mu_0\bRL-\beta\mu_0\bRL\right.\nonumber\\
	&&\;\;\;\;\;\;\;\;\;\;\;\;\;\;\;\;\;\;\;\;\;\;\;\;\;\;\;\;\;\;\;\;\;\;\;\;\;\;\;\;\;\;\;\;\;\;\;\;\;\;\;\;\;\;\;\;\;\;\;\;\;\;\;\;\;\;\;\;\;\;\;\;\;\;\;\;\;\;\;\;\;\left.+\beta p_{\Lambda,\beta,\zero}(\mu_0)\right\}\nonumber\\
	& = & \sup_{\mu'}\left\{\beta\mu'\tRL-\beta p_{\Lambda,\beta,\zero}(\mu')-[\beta\mu_0\bRL-\beta p_{\Lambda,\beta,\zero}(\mu_0)]+\beta\mu_0(\bRL-\tRL)\right\}\nonumber\\
	& = & \beta f_{\Lambda,\beta,\zero}^{GC}(\tRL)-\beta f_{\Lambda,\beta,\zero}^{GC}(\bRL)+\beta\mu_0(\bRL-\tRL),
	\label{FV-I}
	\end{eqnarray}
where $\mu'=\mu+\mu_0$ and we used the grand-canonical free energy define in \eqref{GCFE1}.

Note that we have also used the fact that 
$\bar\rho_\Lambda=p_{\Lambda,\beta,\zero}'(\mu_0)$ and
$\beta f_{\Lambda,\beta,\zero}^{GC}(\bar{\rho}_{\Lambda})=\beta\mu_0\bar{\rho}_{\Lambda}-\beta p_{\Lambda,\beta,\zero}(\mu_0)$. In the limit $|\Lambda|\rightarrow\infty$ we obtain (for the moment assuming that $f_{\Lambda,\beta,\zero}^{GC}\to f_\beta$ and $\bar\rho_\Lambda\to\rho_0$)

	\begin{equation}\begin{split}
	\lim_{|\Lambda|\rightarrow\infty}\frac{1}{|\Lambda|}\log\mathbb{P}^{\zero}_{\Lambda,\mu_0}\left(A_{\tN}\right)=- I_\beta(\tilde\rho;\rho_0),
	\end{split}
	\label{Classical-LDP}
	\end{equation} 
where $\rho_0$ is given in \eqref{MU} and
\begin{equation}
I_\beta(\tilde\rho;\rho_0):=\beta f_{\beta}(\tilde\rho)-\beta f_{\beta}(\rho_0)-\beta f'_{\beta}(\rho_0)(\tilde\rho-\rho_0).
\label{LD-F}
\end{equation}

	Next, one can go a step further and study the local moderate deviations ($\alpha\in[1/2, 1)$ in \eqref{GeneralDeviations}) by taking the Taylor expansion of \eqref{I} around $\bRL|\Lambda|$. Doing this we will find that  $\mathcal{I}^{\zero}_{\Lambda,\beta,\mu_0}(\bRL|\Lambda|)$ and $(\mathcal{I}^{\zero}_{\Lambda,\beta,\mu_0})'(\bRL|\Lambda|)$ are equal to zero. This happens because, using the fact that $L^{\zero}_{\Lambda,\beta,\mu_0}(\mu)$ is a strictly convex function of $\mu$, the supremum in \eqref{I} is obtained at $\mu=0$ when we consider $\bRL|\Lambda|$ instead of $\tN$. 
	Hence, we get:
	\begin{equation}
	\mathcal{I}^{\zero}_{\Lambda,\beta,\mu_0}(\tN)=\frac{(\tN-\bRL|\Lambda|)^2}{2|\Lambda|\sigma^2_{\Lambda,\zero}(\mu_0)}+\sum_{j\ge3}\frac{Q^{(j)}_{\Lambda,\zero}}{j!}\left(\frac{\tN-\bRL|\Lambda|}{|\Lambda|}\right)^j,
	\label{IT}
	\end{equation}
	where the coefficients $Q^{(j)}_{\Lambda,\zero}$ are polynomials which can be computed via the momenta (see eq. (1.2.20)-(1.2.23) of \cite{dobrushin1994large}) and where we used \eqref{FV-I} and the fact that

	\begin{equation}
	(f^{GC}_{\Lambda,\beta,\zero})''(\bRL)=\frac{1}{ p''_{\Lambda,\beta,\zero}(\mu_0)}=\frac{1}{\beta\sigma^2_{\Lambda,\zero}(\mu_0)}.
	\label{GC-Variance}
	\end{equation}

In what follows we explain our strategy working with the canonical partition function.
As it will be argued later in detail, having the cluster expansion of the canonical ensemble, one can compute a given deviation directly without need of following the above strategy.
From \eqref{GcProbE} we have
\begin{equation}
\probBC(A_{\tN})=\frac{e^{\beta\mu_0\tN}\canBC(\tN)}{\GcanBC(\mu_0)},
\label{GcProbE1.1}
\end{equation}	
which can be rewritten as 
\begin{equation}
\probBC (A_{\tN})=J^{C}_{\mu_0}(\tN,\bN)K(\mu_0,\bN).
\label{GcProbE1}
\end{equation}
Here $\bar N_\Lambda$ is given by \eqref{MeanValue} and for $\mu\in\R$ and $N,N'\in\mathbb{N}$ we defined
\begin{equation}
J^{C}_{\mu}(N,N'):=\frac{e^{\beta\mu N}\canBC(N)}{e^{\beta\mu N'}\canBC(N')}
\label{NUM-C}
\end{equation}
and
\begin{equation}
K(\mu,N):=\left(\frac{\GcanBC(\mu)}{e^{\beta\mu N}\canBC(N)}\right)^{-1}.
\label{Def-Normalizzazione}
\end{equation} 
The strategy is similar to the one before and it consists of perturbing around $\bN$, however we will see next that we have to slightly vary this choice.
From \eqref{FreeE} the term $J^{C}_{\mu_0}(\tN,\bN)$ has the following form 

	\begin{equation}
	J^{C}_{\mu_0}(\tN,\bN)=\exp\left\{ \beta\mu_0(\tN-\bN)+|\Lambda|\beta f_{\Lambda,\beta,\zero}(\bN)-|\Lambda|\beta f_{\Lambda,\beta,\zero}(\tN)\right\},
	\label{J}
	\end{equation}
which is the finite volume version of \eqref{LD-F} viewed in the canonical ensemble.
Furthermore,  as we will see in the sequel, working with the canonical partition function one can also perform a direct calculation for the pre-factor, since from \eqref{GrandCan0} and \eqref{Def-Normalizzazione} we have 

\begin{equation}\label{KwithJ}
[K(\mu_0,\bar N_\Lambda)]^{-1}=\sum_{N\ge0}J^{C}_{\mu_0}(N,\bN).
\end{equation}

	For the moderate deviations, in order to compute $J^C_{\mu_0}(\tN,\bN)$ and $K(\mu_0,\bN)$, we will use the free energy $\F$ defined in \eqref{FreeERL1} and related to the one defined in  \eqref{FreeE} via \eqref{A}. Hence, performing a Taylor expansion around $\bRL$  we obtain
\begin{equation}\begin{split}
\bigg|\frac{\log J^{C}_{\mu_0}(\tilde {N},\bN)}{|\Lambda|}-\bigg[\beta(\mu_0-\F'(\bRL))(\tRL-\bRL)-\F''(\bRL)\frac{(\tRL-\bRL)^2}{2}
\\
+o((\tRL-\bRL)^2)\bigg]\bigg|\s\frac{\log\sqrt{|\Lambda|}}{|\Lambda|}\nonumber
\end{split}
\end{equation}
where now $\tRL-\bRL=u/|\Lambda|^{1/2}$. 
As we expect the term $\F''(\bRL)\frac{(\tRL-\bRL)^2}{2}$ to be dominant, we find ourselves in trouble since a rough estimate from
\cite{pulvirenti2015finite} (see also Appendix \ref{appendice1}) for the finite volume corrections of the free energy - also recalling from \eqref{MU} that $\mu_0=f_\beta'(\rho_0)$ - gives that
\begin{equation}\label{rough}
(\mu_0-\F'(\bRL))(\tRL-\bRL)\sim\frac{|\partial\Lambda|}{|\Lambda|^{3/2}}>>\frac{1}{|\Lambda|}\sim\F''(\bRL)\frac{(\tRL-\bRL)^2}{2}.\nonumber
\end{equation}
The remedy will come from the fact that we can get an improved estimate when the finite volume estimate is done at the density which corresponds to the supremum of the ``canonical" Legendre transform. As before, we call this density $\rho^*_{\Lambda}=N^*/|\Lambda|$, where $N^*$ is the number of particles where the supremum at equation \eqref{N*} occurs when $\mu=\mu_0$. Note that this is similar with what happens in the grand canonical with the difference that now we should not perturb around the ``grand-canonical" $\bar N_\Lambda$, but instead, around its slightly different ``canonical" counterpart ($N^*$). 

As we will see later in Remark \ref{RemPre} we have that
\begin{equation} \lim_{|\Lambda|\rightarrow\infty}\frac{N^*}{|\Lambda|}=\lim_{|\Lambda|\rightarrow\infty}\frac{\bN}{|\Lambda|}=\rho_0.
\label{lim}
\end{equation}	
Moreover,  $\rho^*_{\Lambda}$ gives us an expression of the chemical potential in terms of canonical partition function and canonical free energy. Indeed from \eqref{A} and \eqref{N*}  we observe that the function
\begin{equation}
\rho\mapsto \F(\rho)+S_{|\Lambda|}(\rho),
\end{equation}
has a maximum at $\rho^*_{\Lambda}$. Thus,
\begin{equation}
\mu_0=\F'(\rho^*_{\Lambda})+S'_{|\Lambda|}(\rho^*_{\Lambda}),
\label{mu_0}
\end{equation} 
with $S'_{|\Lambda|}(\rho^*_{\Lambda})$ given by \eqref{Stirling1} and it is such that $S'_{|\Lambda|}(\rho^*_{\Lambda})\;\lesssim\;1/|\Lambda|$. 
Equation \eqref{mu_0} implies that the system at finite volume pushes us to consider as center of deviation the value $N^*$ instead of $\bN$. Relation \eqref{lim} implies that the density $N^*$ preferred by the system has the same limit as $\bN$. But at finite volume, it could happen that it is distant from $\bN$ more than $|\Lambda|^{1/2}$ making the study of small fluctuations irrelevant. As we prove in Lemma \ref{LemmaBN} this is not the case. 


\section{Proofs of the Theorems}
\label{ProofSec}

In this section we give the proofs of the main results, which are based on some technical Lemmas presented in the next Section. Moreover, we will give some remarks to clarify the relations between the canonical, gran-canonical and thermodynamic quantities involved in the problem and presented above.  

\begin{proof} [Proof of Theorem \ref{Th1}]
	We rewrite $\probBC(A_{\tN})$ as follows:
	\begin{equation}\label{option1}
	\probBC(A_{\tN})=\frac{\GcanBC(\tilde{\mu}_{\Lambda})e^{\beta\mu_0\tN}}{\GcanBC(\mu_0)e^{\beta\tilde{\mu}_{\Lambda}\tN}}\mathbb{P}^{\zero}_{\Lambda,\tilde{\mu}_{\Lambda}}(A_{\tN}).
	\end{equation}
	In the previous one we did the Radon-Nikod\'ym derivative of our probability measure with respect to the one with $\tilde{\mu}_{\Lambda}$ instead of $\mu_0$. Note that the definition of $\tilde{\mu}_{\Lambda}$ given via \eqref{GCFE1}, i.e., such that 
	\begin{equation}
	\beta\fgc(\tRL)=\beta \tilde{\mu}_{\Lambda}\tRL-\beta p_{\Lambda,\beta,\zero}(\tilde{\mu}_{\Lambda}),
	\label{f-GC_Tmu}
	\end{equation}
	is equivalent to define implicitly $\tilde{\mu}_{\Lambda}$ as the chemical potential such that
	\begin{equation}
	\frac{\tilde{N}}{|\Lambda|}=\mathbb{E}^{\zero}_{\Lambda,\tilde{\mu}_{\Lambda}}\left[\frac{N}{|\Lambda|}\right]=\frac{\partial}{\partial\mu}p_{\Lambda,\beta,\zero}(\mu)\bigg|_{\mu=\tilde{\mu}_{\Lambda}}.
	\label{Tilde-mu}
	\end{equation}
	Moreover, from \eqref{relation1} and \eqref{FV-I} we have that this $\tilde{\mu}_{\Lambda}$ is equal to the one which satisfies \eqref{I}.

	From \eqref{FVP}, \eqref{GCFE1}, \eqref{LDO} and \eqref{f-GC_Tmu}  we get
	\begin{eqnarray}
	\frac{\GcanBC(\tilde{\mu}_{\Lambda})e^{\beta\mu_0\tN}}{\GcanBC(\mu_0)e^{\beta\tilde{\mu}_{\Lambda}\tN}} &= &\exp\left\{|\Lambda|\left[\beta\mu_0\tN-\beta\tilde{\mu}_{\Lambda}\tN+\beta p_{\Lambda,\beta,\zero}(\tilde{\mu}_{\Lambda})-\beta p_{\Lambda,\beta,\zero}(\mu_0)\pm\beta\mu_0\bN\right]\right\}\nonumber
	\\
	&=&\exp\left\{|\Lambda|\left[\beta \fgc(\bRL)-\beta\fgc(\tRL)+\beta\mu_0(\tRL-\bRL)\right]\right\}\nonumber
	\\
	&=& \exp\left\{- |\Lambda| I^{GC}_{\Lambda,\beta,\zero}(\tRL;\bRL)\right\}.
	\end{eqnarray}
	On the other hand, denoting with $\tN^*$ the number of particles such that 
	\begin{equation}
	\sup_{N}\left\{e^{\beta\tilde{\mu}_{\Lambda}N}\canBC(N)\right\}=e^{\beta\tilde{\mu}_{\Lambda}\tN^*}\canBC(\tN^*),
	\end{equation}
	using \eqref{NUM-C} and \eqref{Def-Normalizzazione} we have
	\begin{equation}
	\mathbb{P}^{\zero}_{\Lambda,\tilde{\mu}_{\Lambda}}(A_{\tN})=J^{C}_{\tilde{\mu}_{\Lambda}}(\tN,\tN^*)K(\tilde{\mu}_{\Lambda},\tN^*).
	\end{equation}
	The novelty here is that we compute the above term using cluster expansions instead
	of inverting the characteristic function as in \eqref{invert}.
	First, we note that from Lemma \ref{LemmaBN} we have
	\begin{equation}\label{est}
	|\tN-\tN^*|\le C,
	\end{equation}
	for some $C>0$ which does not depend on $\Lambda$.
	Then, applying Lemma \ref{COR1} (using the cluster expansion \eqref{CanCE1}) we find
	\begin{eqnarray}
	J^{C}_{\tilde{\mu}_{\Lambda}}(\tN,\tN^*)&=&\exp\left\{ S'_{|\Lambda|}(\tRL^*)(\tN-\tN^*)-\sum_{m\ge 2}\frac{(\tN-\tN^*)^{m}}{|\Lambda|^{m-1}}\frac{\F^{(m)}(\tRL^*)}{m!}+ |\Lambda|S_{|\Lambda|}(\tRL^*)\right\}\nonumber
	\\
	&\s &\exp\left\{|\Lambda|S_{|\Lambda|}(\tRL^*)\right\}\left(1+\frac{1}{|\Lambda|}\right),
	\label{Eq1LD}
	\end{eqnarray}
	since \eqref{est} and \eqref{Stirling1} and where $S_{|\Lambda|}(\rho^*_{\Lambda})$ is given by \eqref{StirlingEq} with the property \eqref{StirlingBound}.
	
	The study of $K(\tilde{\mu}_{\Lambda},\tN^*)$ is the same as the one done in Lemma \ref{Lemma 3} where now we consider $\tN^*$ as center of fluctuations of order 1/2.
	Hence the conclusion follows from 
	\begin{eqnarray}
	K(\tilde{\mu}_{\Lambda},\tN^*)\le e^{-|\Lambda|S_{|\Lambda|}(\tRL^*)}\left[\sqrt{2\pi D_{\Lambda,\zero}(\tRL^*)|\Lambda|}\left(1-\frac{C}{\sqrt{|\Lambda|}}\right)\right]^{-1}
	\end{eqnarray}	
	and
	\begin{equation}
	K(\tilde{\mu}_{\Lambda},\tN^*)\ge e^{-|\Lambda|S_{|\Lambda|}(\tRL^*)}\left[\sqrt{2\pi D_{\Lambda,\zero}(\tRL^*)|\Lambda|}\left(1+\frac{C}{\sqrt{|\Lambda|}}\right)\right]^{-1}
	\end{equation}
	for some $C\in\R^+$ independent on $\Lambda$.
\end{proof}

\begin{remark} In the proof of Theorem \ref{Th1}, instead of \eqref{option1} we could try with the canonical partition function and obtain:
	\begin{equation}\label{option2}
	\probBC(A_{\tN})=J^C_{\mu_0}(\tN,N^*)K(\mu_0,N^*),
	\end{equation}
	where
	\begin{equation}
	J^C_{\mu_0}(\tN,N^*)=\exp\left\{-|\Lambda|\left[\beta f_{\Lambda,\beta,\zero}(\tN)-\beta f_{\Lambda,\beta,\zero}(N^*)-\beta\mu_0(\tN-N^*)\right]\right\}.
	\label{Canonical_LDP}
	\end{equation}
	In this case, the normalization $K(\mu_0,N^*)$ can only be given in terms of the full series
	of derivatives $\F^{(m)}(\rho^*_{\Lambda})$, $m\geq 2$ giving:
	\begin{eqnarray}
	K(\mu_0,N^*) & = & \sum_{N\in I}\exp\left\{-\frac{(N-N^*)^2}{2|\Lambda|}\left[-2\beta b_{|\Lambda|}(\rho^*_{\Lambda})\frac{|\Lambda|}{N-N^*}+\rho^*_{\Lambda}B_{|\Lambda|}(\rho^*_{\Lambda})\left(\frac{|\Lambda|}{N-N^*}\right)^2\right.\right.\nonumber
	\\
	&& \left.\left.+\beta \F''(\rho^*_{\Lambda})+\sum_{m\ge3}\frac{2\F^{(m)}(\rho^*_{\Lambda})}{m!}\left(\frac{N-N^*}{|\Lambda|}\right)^{m-2}\right]\right\}+O(e^{-c|\Lambda|}  ).
	\label{Norm-LD}
	\end{eqnarray}
	Instead, when we perturbed around $\tilde N^*$ in $K(\tilde{\mu}_{\Lambda},\tN^*)$ it implied that the second derivative is dominant and hence
	written as a Gaussian integral.
\end{remark}
\begin{remark}
	Let us note that equations \eqref{LDO} and \eqref{Canonical_LDP} give the same infinite volume functional. 
	For the canonical ensemble, from \cite{pulvirenti2015finite} we have that
	\begin{equation}
	| I_{\beta}(\tilde{\rho};\rho_0)-I^{C}_{\Lambda,\beta,\zero}(\tN;N^*)|\;\s\; \frac{|\partial\Lambda|}{|\Lambda|},
	\end{equation}
	where $I^{C}_{\Lambda,\beta,\zero}(\tN;N^*):=\beta f_{\Lambda,\beta,\zero}(\tN)-\beta f_{\Lambda,\beta,\zero}(N^*)-\beta\mu_0(\tN-N^*)$.
	Instead, using the fact that Lemma \ref{Lemma 3} and Lemma \ref{LemmaBN} imply
	\begin{eqnarray}
	\left|\beta\fgc(\hat{\rho}_{\Lambda})-\beta f_{\Lambda,\beta,\zero}(\hat{N}^*)\right|& \le&  \frac{1}{|\Lambda|}\log\left[\frac{\sum_{N\ge0}e^{\beta\hat{\mu} N}\canBC(N)}{e^{\beta\hat{\mu} \hat{N}^*}\canBC(\hat{N}^*)}\right]\nonumber
	\\
	&+&\beta \hat{\mu}\frac{\hat{\rho}_{\Lambda}-\hat{\rho}^*_{\Lambda}}{|\Lambda|}\nonumber
	\;\s\;\frac{\log\sqrt{|\Lambda|}}{|\Lambda|},
	\end{eqnarray}
	for all $\hat{\mu},\;\hat{N}$ and $\hat{N}^*$ related in the sense of \eqref{MeanValue} and \eqref{N*}, we get
	\begin{equation}
	| I^{GC}_{\Lambda,\beta,\zero}(\tRL;\bRL)-I^{C}_{\Lambda,\beta,\zero}(\tN;N^*)|\;\s\; \frac{\log{\sqrt{|\Lambda|}}}{|\Lambda|}.
	\end{equation}
	
	Furthermore, we have 
	\begin{equation}
	\left|D_{\Lambda,\zero}(\tRL^*)-\sigma^2_{\Lambda,\zero}(\tilde{\mu}_{\Lambda})\right|\;\s\;\frac{1}{\sqrt{|\Lambda|}}.
	\label{ConfrontoVarianze}
	\end{equation}
	Similarly, denoting $\sigma^2_{\infty}(\tilde{\mu}):=\frac{1}{\beta} p''_{\beta}(\tilde{\mu})= \lim_{\Lambda\rightarrow\R^d}\sigma^2_{\Lambda,\zero}(\tilde{\mu}_{\Lambda})$, we obtain that
	\begin{equation}
	|D_{\Lambda,\zero}(\tRL^*)-\sigma^2_{\infty}(\tilde{\mu})|\;\s\;\frac{|\partial\Lambda|}{|\Lambda|},
	\label{SigmaINF}
	\end{equation}
	thanks to Lemma \ref{LemmaA1} and the fact that, from \eqref{FreeELT}  and the equivalent formulation of \eqref{MU} for $\tilde{\mu}$ and $\tilde{\rho}$, we have 
	\begin{equation}
	f''_{\beta}(\tilde{\rho})=\frac{\beta}{\sigma^2_{\infty}(\tilde{\mu})},
	\end{equation}
	where $\tilde{\rho}:=\lim_{\Lambda\rightarrow\R^+}\tilde{\rho}_{\Lambda}$.
	\label{Remark1}
\end{remark}	

\begin{proof} [Proof of Theorem \ref{Th2}] 
	From \eqref{NUM-C}, \eqref{Def-Normalizzazione} we have 
	\begin{equation}
	\probBC(A_{\tN})=J^{C}_{\mu_0}(\tN,N^*)K(\mu_0,N^*).
	\end{equation}
	Then using Lemma \ref{COR1} we have 
	\begin{equation}
	J^{C}_{\mu_0}(\tN,N^*)\ge\exp\left\{-\frac{(u')^2|\Lambda|^{2\alpha-1}}{2D^{\alpha}_{\Lambda,\beta,\zero}(\rho^*_{\Lambda})}+|\Lambda|S_{|\Lambda|}(\rho^*_{\Lambda})-E_{|\Lambda|}(\alpha,u',\rho^*_{\Lambda})\right\}
	\end{equation}
	and
	\begin{equation}
	J^{C}_{\mu_0}(\tN,N^*)\le \exp\left\{-\frac{(u')^2|\Lambda|^{2\alpha-1}}{2D^{\alpha}_{\Lambda,\beta,\zero}(\rho^*_{\Lambda})}+|\Lambda|S_{|\Lambda|}(\rho^*_{\Lambda})+E_{|\Lambda|}(\alpha,u',\rho^*_{\Lambda})\right\}
	\end{equation}
	where $S_{|\Lambda|}(\rho^*_{\Lambda})$ given by \eqref{StirlingEq} with the property \eqref{StirlingBound}.
	
	The conclusion follows from Lemma \ref{Lemma 3} which gives us
	\begin{eqnarray}
	K(\mu_0,N^*)\le e^{-|\Lambda|S_{|\Lambda|}(\rho^*_{\Lambda})}\left[\sqrt{2\pi D^{\alpha,+}_{\Lambda,\zero}(\rho^*_{\Lambda})|\Lambda|}\left(1-E_{|\Lambda|}(\alpha,u',\rho^*_{\Lambda})\right)\right]^{-1}
	\end{eqnarray}	
	and
	\begin{equation}
	K(\mu_0,N^*)\ge e^{-|\Lambda|S_{|\Lambda|}(\rho^*_{\Lambda})}\left[\sqrt{2\pi D^{\alpha,+}_{\Lambda,\zero}(\rho^*_{\Lambda})|\Lambda|}\left(1+E_{|\Lambda|}(\alpha,u',\rho^*_{\Lambda})\right)\right]^{-1}.
	\end{equation}
\end{proof}

\begin{remark}
	Note that, similarly to Remark 2.2, we have
	\begin{equation}
	\left|D^{\alpha}_{\Lambda,\zero}(\rho^*_{\Lambda})-\sigma^2_{\Lambda,\zero}(\mu_0)\right|\;\s\;\frac{1}{|\Lambda|^{m(\alpha)(1-\alpha)-1/2}}.
	\end{equation}
	and 
	\begin{equation}
	|D_{\Lambda,\zero}(\rho^*_{\Lambda})-\sigma^2_{\infty}(\mu_0)|\;\s\;\frac{|\partial\Lambda|}{|\Lambda|}.
	\end{equation}
\end{remark}

\begin{proof}[Proof of Corollary \ref{Corollary}]
	The proof follows from the proof of Theorem \ref{Th2} for $\alpha=1/2$.
\end{proof}

\section{Technical Lemmas}
\label{S2}	
In this section we give the technical details for the proofs of the main theorems. We write the canonical finite volume free energy \eqref{FreeE} calculated at $\tN$, as a Taylor expansion around $N^*$ using the free energy defined in \eqref{FreeERL1}. In order to simplify the notation we do not explicit the dependence on the boundary conditions in this first part of the section.

We recall that, using the cluster expansion \eqref{CanCE1}, the free energy \eqref{FreeE} can be written as  
\begin{equation}
f_{\Lambda,\beta}(N)=\frac {1}{\beta}\left\{-\frac{1}{|\Lambda|}\log\frac{|\Lambda|^N}{N!}-F^{(int)}_{\Lambda,\beta}(N)\right\}
\end{equation}
where, using  \eqref{Cluster-Coefficient}, we defined
\begin{equation}
F^{(int)}_{\Lambda,\beta}(N):=\frac{N}{|\Lambda|}\sum_{n\ge 1}\frac{1}{n+1}P_{N,|\Lambda|}B_{\Lambda,\beta}(n).
\end{equation}
Moreover, defining
\begin{equation}
\mathcal{F}^{(int)}_{\Lambda,\beta}(\rho):=\sum_{n\ge1}\frac{1}{n+1}\mathcal{P}_{n+1}(\rho)B_{\Lambda,\beta}(n),
\end{equation}
the free energy defined in \eqref{FreeERL1} can be written as
\begin{equation}
\mathcal{F}_{\Lambda,\beta}(\rho)=\frac{1}{\beta}\left\{\rho(\log\rho-1)-\mathcal{F}^{(int)}_{\Lambda,\beta}(\rho)\right\},
\end{equation}
We also recall that, being $F^{(int)}_{\Lambda,\beta}(N)=\mathcal{F}^{(int)}_{\Lambda,\beta}(\rho_{\Lambda})$ ($N\in\mathbb{N},\;\rho_{\Lambda}=N/|\Lambda|$), between the free energy defined in \eqref{FreeE} and the one defined in \eqref{FreeERL1}, from \eqref{StirlingEq} and \eqref{StirlingBound}, it holds the following relation:
\begin{equation}
\left|f_{\Lambda,\beta}(N)-\mathcal{F}_{\Lambda,\beta}(\rho_{\Lambda})\right|=|S_{|\Lambda|}(\rho_{\Lambda})|\;\s\;\frac{\log{\sqrt{|\Lambda|}}}{|\Lambda|}.
\label{FreeEn-MOD}
\end{equation}

In what follows we will denote with $\f^{int,(m)}_{\Lambda,\beta}(\cdot)$ and $\mathcal{P}^{(m)}_{n+1}(\cdot)$ the $m$-th derivative of
$\f^{int}_{\Lambda,\beta}(\cdot)$ and $\mathcal{P}_{n+1}(\cdot)$.

The following result holds:
\begin{lemma}
	Let  $N$, $N'$ be such that $\rho_{\Lambda},\rho'_{\Lambda}\in(0,1)$ and equation \eqref{CanCE1} is true for both $N$ and $N'$. 
	Then:
	\begin{equation}\begin{split}
	\bigg|f_{\Lambda,\beta}(N)-\left[f_{\Lambda,\beta}(N')+\sum_{m\ge1}\left(\frac{N-N'}{|\Lambda|}\right)^m\frac{\f^{(m)}_{\Lambda,\beta}(\rho'_{\Lambda})}{m!}\right]\bigg|\s \frac{\log\sqrt{|\Lambda|}}{|\Lambda|}.
	\end{split}
	\end{equation}	
	\label{Lemma 1}	
\end{lemma}	
\begin{proof} Let us fix $n\in\mathbb{N}$ and define for all $k\le n$ the set $\{i_1,...,i_k\}_{\ne}:=\{\{i_1,...i_k\}\subset\{1,...,n\}\;|\;i_s\ne i_t\;\forall\;1\le s,t\le k,\;s\ne t\}$. For all $\rho_{\Lambda}$ we have:

	\begin{eqnarray}
	\frac{1}{m!}\p^{(m)}_{n+1}(\rho_{\Lambda}) &= & {{n+1}\choose m}\rho_{\Lambda}^{n+1-m}\times\nonumber
	\\
	& \times & \left[1+\sum_{k=1}^{n+1-m}(-1)^k  \frac{{{n+1}\choose m}^{-1}{{n+1-k}\choose m}}{(\rho_{\Lambda}|\Lambda|)^k}\sum_{\{i_1,...,i_k\}_{\ne}\subseteq\{0,...,n-1\}}\prod_{j=1}^k(n-i_j)\right]\nonumber
	\label{P-forma2}
	\end{eqnarray}
	and
	\begin{equation}\begin{split}
	P_{\lfloor\rho_{\Lambda}|\Lambda|\rfloor,|\Lambda|}(n)=\rho_{\Lambda}^n\left[1+\sum_{k=1}^n(-1)^k\frac{\sum_{\{i_1,...,i_k\}_{\ne}\subseteq\{1,...,n-1\}}\prod_{j=1}^k(n-i_j)}{(\rho_{\Lambda}|\Lambda|)^k}\right].\nonumber
	\end{split}
	\label{P-forma1}
	\end{equation}
	Noting that 
	\begin{equation}\begin{split}
	&\left[1 - \frac{\sum_{i=0}^{n-1}(n-i)}{\rho_{\Lambda}|\Lambda|}\left(2-{{n+1}\choose m}^{-1}{n\choose m}\right)\right]
	\\
	& +\sum_{\substack{k=2\\k\;\mathrm{even}}}^n\left[\frac{\sum_{\{i_1,...,i_k\}_{\ne}\subseteq\{0,...,n-1\}}\prod_{j=1}^k(n-i_j)}{(\rho_{\Lambda}|\Lambda|)^k}\left(2-{{n+1}\choose m}^{-1}{n+1-k\choose m}\right)\right.
	\\
	&\left.-\frac{\sum_{\{i_1,...,i_{k+1}\}_{\ne}\subseteq\{0,...,n-1\}}\prod_{j=1}^{k+1}(n-i_j)}{(\rho_{\Lambda}|\Lambda|)^{k+1}}\left(2-{{n+1}\choose m}^{-1}{n+1-(k+1)\choose m}\right)\right]\ge0\nonumber
	\end{split}
	\end{equation}
	we get
	\begin{equation}
	\left[\frac{1}{m!}\p_{n+1}^{(m)}(\rho_{\Lambda})\right]
	\bigg[P_{\lfloor\rho_{\Lambda}|\Lambda|\rfloor,|\Lambda|}(n)\bigg]^{-1}\le 2{{n+1}\choose m}\rho_{\Lambda}^{1-m}.
	\label{CovClusrDeriv}
	\end{equation}
	Thanks to the previous bound, using \eqref{absCan} and Stirling's formula we obtain
	\begin{equation}\begin{split}
	\bigg|\sum_{m\ge1}\sum_{n\ge m-1}\frac{1}{n+1}\frac{1}{m!}\p_{n+1}^{(m)}(\rho_{\Lambda})B_{\Lambda,\beta}(n)\bigg|
	\le 2\sum_{m\ge1}\rho_{\Lambda}^{1-m}\sum_{n\ge m-1}{{n+1}\choose m}|F_{\beta,N,\Lambda}(n)|
	\\
	\le 2\sum_{m\ge 1}\frac{1}{m!}\rho_{\Lambda}^{1-m}\sum_{n\ge m-1}(n+1)^m e^{-c(n+1)}<\infty.
	\end{split}
	\label{Convergence}
	\end{equation}
	
	From \eqref{PNL} and \eqref{PNL1} it is easy to see by induction that for all $N$ and $N'$, the term $(N/|\Lambda|)P_{N,|\Lambda|}(n)$ can be written as:
	
	\begin{eqnarray}
	\frac{N}{|\Lambda|}P_{N,|\Lambda|}(n)&=&\frac{[(N-N')+N'][(N-N')+(N'-1)]\cdot\cdot\cdot[(N-N')+(N'-n)]}{|\Lambda|^{n+1}}\nonumber
	\\
	&=&\sum_{m=1}^{n+1}\frac{1}{m!}\left(\frac{N-N'}{|\Lambda|}\right)^m\p^{(m)}_{n+1}(\rho'_{\Lambda})+\frac{N'}{|\Lambda|}P_{N',|\Lambda|}(n).
	\label{EXP1}
	\end{eqnarray} 
	Then using \eqref{FreeERL1} and \eqref{EXP1} and thanks to \eqref{Convergence} we have
	\begin{eqnarray}
	\frac{N}{|\Lambda|}\sum_{n\ge1}\frac{1}{n+1}P_{N,|\Lambda|}(n)\Bbl(n)&=&-\beta\left\{\sum_{m\ge1}\left(\frac{N-N'}{|\Lambda|}\right)^m\frac{1}{m!}\f^{int,(m)}_{\Lambda,\beta}(\rho'_{\Lambda})\right\}\nonumber
	\\
	&&+\frac{N'}{|\Lambda|}\sum_{n\ge1}\frac{1}{n+1}P_{N',|\Lambda|}(n)\Bbl(n).
	\label{KeyIdea}
	\end{eqnarray}
	Observing now that the Taylor expansion of  $\rho_{\Lambda}(\log\rho_{\Lambda}-1)$ around $\rho'_{\Lambda}$ is equal to
	\begin{eqnarray}
	\rho_{\Lambda}(\log\rho_{\Lambda}-1)& = &\rho'_{\Lambda}(\log\rho'_{\Lambda}-1)+(\rho_{\Lambda}-\rho'_{\Lambda})\log\rho'_{\Lambda}\nonumber
	\\
	& + & \sum_{m\ge2}(-1)^m\frac{(\rho_{\Lambda}-\rho'_{\Lambda})^m}{m!}\frac{(m-2)!}{(\rho'_{\Lambda})^{m-1}}.
	\label{log}
	\end{eqnarray}
	using \eqref{FreeEn-MOD}, \eqref{KeyIdea} and \eqref{log} we conclude the proof.
\end{proof}

As a consequence of the previous lemma,
for the term $J^C_{\mu}(N,N')$  given in \eqref{NUM-C} we have:
\begin{lemma}
Let $\mu_0\in\R$ be a chemical potential  and $N^*$ which satisfies \eqref{N*} for $\mu=\mu_0$ such that condition $(\star)$ holds. 
	For a generic fluctuation $N$ such that
	\begin{equation}
	N:=N^*+v|\Lambda|^{\alpha}
	\end{equation}
	for some $v\in\R$ and $\alpha\in[1/2,1)$ so that $N\in\mathbb{N}$
	for the quantity $J_{\mu_0}^C(N,N^*)$ defined in \eqref{NUM-C} we have:
	\begin{equation}
	J_{\mu_0}^C(\tN,N^*)\le\exp\left\{-\frac{v^2|\Lambda|^{2\alpha-1}}{2D^{\alpha}_{\Lambda,\zero}(\rho^*_{\Lambda})}+|\Lambda|S_{|\Lambda|}(\rho^*_{\Lambda})+E_{|\Lambda|}(\alpha,v,\rho^*_{\Lambda})\right\}
	\label{SviluppoJ1}
	\end{equation}
	and
	\begin{equation}
	J_{\mu_0}^C(\tN,N^*)\ge\exp\left\{-\frac{v^2|\Lambda|^{2\alpha-1}}{2D^{\alpha}_{\Lambda,\zero}(\rho^*_{\Lambda})}+|\Lambda|S_{|\Lambda|}(\rho^*_{\Lambda})-E_{|\Lambda|}(\alpha,v,\rho^*_{\Lambda})\right\}
	\label{SviluppoJ2}
	\end{equation}
	where, setting
	\begin{equation}
	m(\alpha):=\min\left\{m\in\mathbb{N}\;|\;m(1-\alpha)-1>0\right\},
	\label{m-Coda}
	\end{equation}
	$D^{\alpha}_{\Lambda,\zero}(\rho^*_{\Lambda})$ is defined in \eqref{Var2},
	$S_{|\Lambda|}(\rho^*_{\Lambda})$ is given by \eqref{StirlingEq} with the property \eqref{StirlingBound}, and where $E_{|\Lambda|}(\alpha,v,\rho^*_{\Lambda})$ is an error term of order $|\Lambda|^{-[m(\alpha)(1-\alpha)-1]}$, which will be given in \eqref{Error}. 
	\label{COR1}
\end{lemma}

\begin{proof}

	From \eqref{J} and  Lemma \ref{Lemma 1}, i.e., doing the Taylor expansion of $f_{\Lambda,\beta,\zero}(\tN)$ around $N^*$ in the sense of Lemma \ref{Lemma 1}, we obtain 
	\begin{eqnarray}
	J^C_{\mu_0}(\tN,N^*) & = & \exp\bigg\{\beta v|\Lambda|^{\alpha}(\mu_0-\F'(\rho^*_{\Lambda}))\nonumber\\
	&& -\beta\sum_{m\ge2}\frac{v^m|\Lambda|^{m(\alpha-1)+1}}{m!}\F^{(m)}(\rho^*_{\Lambda})+|\Lambda|S_{|\Lambda|}(\rho^*_{\Lambda})\bigg\},
	\label{Sviluppo0}
	\end{eqnarray}	
	where now
	from \eqref{mu_0} and \eqref{Stirling1} we have
	\begin{equation}
	|\Lambda|^{\alpha}(\mu_0-\F'(\rho^*_{\Lambda}))\;\s\;\frac{1}{|\Lambda|^{1-\alpha}}.
	\end{equation}
	By the definition of $m(\alpha)$ the dominant terms of the sum in \eqref{Sviluppo0} are given by the ones up to $m(\alpha)-1$ where the largest one is given by $m=2$, so that, defining the error as     
	\begin{equation}\begin{split}
	E_{|\Lambda|}(\alpha,v,\rho^*_{\Lambda}):=\frac{\beta}{|\Lambda|^{m(\alpha)(1-\alpha)-1}}\left|\frac{v^{m(\alpha)}\F^{(m(\alpha))}(\rho^*_{\Lambda})}{m(\alpha)!}+\frac{v(\mu_0-\F'(\rho^*_{\Lambda})) }{|\Lambda|^{1-m(\alpha)(1-\alpha)-\alpha}}\right.
	\\
	\left.+\sum_{m\ge m(\alpha)+1}\frac{v^m\F^{(m)}(\rho^*_{\Lambda})}{m!|\Lambda|^{(m-m(\alpha))(1-\alpha)}}\right|,
	\end{split}
	\label{Error}
	\end{equation}
	we can conclude the proof.
\end{proof}	

Now we investigate the term $K(\mu_0,N^*)$ where $\mu_0$ and $N^*$ are related as in \eqref{N*} and \eqref{mu_0}. 

\begin{lemma} 
	
Let $\mu_0\in\R$ be a chemical potential  and $N^*$ which satisfies \eqref{N*} for $\mu=\mu_0$ such that condition $(\star)$ holds. For $K(\mu_0,N^*)$ defined in \eqref{Def-Normalizzazione}, we have 
	\begin{equation}	
	K(\mu_0,N^*)\ge e^{-|\Lambda|S_{|\Lambda|}(\rho^*_{\Lambda})}\left[\sqrt{2\pi D^{\alpha,+}_{\Lambda,\zero}(\rho^*_{\Lambda})|\Lambda|}\left(1+E_{|\Lambda|}(\alpha,v,\rho_{\Lambda}^*)\right)\right]^{-1}
	\end{equation}
	and
	\begin{equation}	
	K(\mu_0,N^*)\le e^{-|\Lambda|S_{|\Lambda|}(\rho^*_{\Lambda})}\left[\sqrt{2\pi D^{\alpha,+}_{\Lambda,\zero}(\rho^*_{\Lambda})|\Lambda|}\left(1-E_{|\Lambda|}(\alpha,v,\rho_{\Lambda}^*)\right)\right]^{-1}
	\end{equation}
	where $D^{\alpha,+}_{\Lambda,\zero}(\rho^*_{\Lambda})$ is defined in \eqref{Var1}, $S_{|\Lambda|}(\rho^*_{\Lambda})$ is given by \eqref{StirlingEq} with the property \eqref{StirlingBound} and $E_{|\Lambda|}(\alpha,v,\rho_{\Lambda}^*)$ error term of order $|\Lambda|^{-[m(\alpha)(1-\alpha)-1]}$ defined via cluster expansion and given by \eqref{Error}, with $m(\alpha)$ given by \eqref{m-Coda}. 
	\label{Lemma 3}
\end{lemma}	

\begin{proof}
	Let us define 
	\begin{equation}\label{set1}
	I_{\alpha,v}:=\mathbb{N}\cap[N^*-v|\Lambda|^{\alpha},N^*+v|\Lambda|^{\alpha}],
	\end{equation}
	with $v\in\R^+$ and $\alpha\in[1/2,1).$
	
	Let also consider the following sets:
	\begin{equation}
	I^c_{\alpha,v}:=\mathbb{N}\setminus I_{\alpha,v}
	\end{equation}
	if $\alpha>(2d-1)/2d$, or
	\begin{equation}\label{set2}
	I^{\alpha,v}_{\delta,v'}:=\mathbb{N}\cap\left\{[N^*-v'|\Lambda|^{\delta},N^*+v'|\Lambda|^{\delta}]\setminus I_{\alpha,v}\right\}
	\end{equation} 
	and
	\begin{equation}\label{set3}
	I^c_{\delta,v'}:=\mathbb{N}\setminus I^{\alpha,v}_{\delta,v'}
	\end{equation}
	if $\alpha\le (2d-1)/2d$ for some $\delta>(2d-1)/2d$ and $v'\in\R^+$.
	
	For all $N\in I^c_{\delta,v'}$ or $N\in I^c_{\alpha,v}$, there exists a positive constant $c_1$ independent on $|\Lambda|$ such that
	\begin{equation}
	J^C_{\mu_0}(N,N^*)=\frac{e^{\beta\mu_0 N}Z^{\zero}_{\beta,\Lambda}(N)}{e^{\beta\mu_0 N^*}Z^{\zero}_{\beta,\Lambda}(N^*)}\le\exp\left\{-c_1\frac{(N-N^*)^2}{|\Lambda|}\right\}.
	\label{GrandiCode1}
	\end{equation}
	Since in these intervals $N$ and $N^*$ are not close we compare with the infinite volume free energy. Hence, adding and subtracting $f_\beta(\rho_\Lambda)$, the exponent of $J^C_{\mu_0}(N,N^*)$ written as in \eqref{J}, becomes:
	\begin{eqnarray}\label{SviluppoCodeGrandiForma1}
	&&-|\Lambda|[f_{\Lambda,\beta,\zero}(N)-f_{\beta}(\rho_{\Lambda})]+|\Lambda|[f_{\Lambda,\beta,\zero}(N^*)-f_{\beta}(\rho^*_{\Lambda})]\nonumber
	\\
	&&-|\Lambda|(\rho_{\Lambda}-\rho^*_{\Lambda})(\mu_0-f'_{\beta}(\rho^*_{\Lambda}))-f''_{\beta}(\hat{\rho}_{\Lambda})\frac{(N-N^*)^2}{2|\Lambda|}
	\label{Jexp}
	\end{eqnarray}
	where we also performed a Taylor expansion of  $f_\beta(\rho_\Lambda)$ around $\rho^*_{\Lambda}$ and where $\hat{\rho}_{\Lambda}\in\left(\min\left\{\frac{N}{|\Lambda|},\rho_0\right\},\max\left\{\frac{N}{|\Lambda|},\rho_0\right\}\right)$. 
	
	From \cite{pulvirenti2015finite} we have that the first two terms of \eqref{Jexp} are of order $|\partial\Lambda|$. For the third term, on one hand form \eqref{MU} and \eqref{fbeta}-\eqref{A.17} we find $\mu_0-f'_{\beta}(\rho^*_{\Lambda})\le|\partial\Lambda|/|\Lambda|$. On the other hand, as it is clarify in \eqref{N-max}, $(\rho_{\Lambda}-\rho^*_{\Lambda})/|\Lambda|\rightarrow 0$ as $|\Lambda|\rightarrow\infty$ for all $N$ here considered. Then also the second term of \eqref{Jexp} has order $|\partial\Lambda|$.
	Moreover we have
	\begin{equation}
	\frac{(N-\bN)^2}{|\Lambda|}\ge|\Lambda|^{2\gamma-1}>|\partial\Lambda|\;\;\;(\gamma=\alpha,\delta)
	\end{equation}
	and
	\begin{equation}
	0<c_1\le\beta f''_{\beta}(\tilde{\rho}_{\Lambda})+C_1\frac{|\partial\Lambda|}{|\Lambda|^{2\gamma-1}}<+\infty\;\;\;(\gamma=\alpha,\delta),
	\label{convexity}
	\end{equation}
	thanks to the fact that we are far from the transition phase, for some $C_1,\;c_1\in\R^+$.
	
	Thus, there exists $c>0$ independents on $|\Lambda|$ such that 
	
	\begin{equation}\begin{split}
	\sum_{N\in I_{\delta,v'}^c(I^c_{\alpha,v})}\frac{e^{\beta\mu^* N}\canBC(N)}{e^{\beta\mu^*N^*}\canBC(N^*)} 
	\le \sum_{N\in I_{\delta}^c(I^c_{\alpha,v})}e^{-c[(N-N^*)^2/|\Lambda|]}\;\s\;  e^{-c|\partial\Lambda|},
	\end{split}
	\label{CodeGrandiForma1}  
	\end{equation}
	
	Let us consider now the sum over $I_{\alpha,v}$. Then, defining   
	\begin{equation}
	D^{\alpha,-}_{\Lambda,\zero}(\rho^*_{\Lambda}):=\left[\beta\F''(\rho^*_{\Lambda})-\beta\sum_{m\ge3}^{m(\alpha)-1}\frac{2v^{m-2}|\F^{(m)}(\rho^*_{\Lambda})|}{m!|\Lambda|^{(m-2)(1-\alpha)}}\right]^{-1},
	\end{equation}
	which is positive thanks to Lemma \ref{LemmaA1} and the fact that $f''_{\beta}(\rho_0)>0$,  and using Lemma \ref{COR1} we have:
	\begin{equation}\begin{split}
	&\sum_{N\in I_{\alpha,v}}\frac{e^{\beta\mu_0 N}\canBC(N)}{e^{\beta\mu_0 N^*}\canBC(N^*)}
	\\
	&\le\exp\left\{E_{|\Lambda|}(\alpha,v,\rho_{\Lambda}^*)+|\Lambda|S_{|\Lambda|}(\rho^*_{\Lambda})\right\}\sum_{n=-v|\Lambda|^{\alpha}}^{v|\Lambda|^{\alpha}}\exp\left\{-\frac{n^2}{2|\Lambda|D^{\alpha,-}_{\Lambda,\zero}(\rho^*_{\Lambda})}\right\}
	\\
	&\le\exp\left\{E_{|\Lambda|}(\alpha,v,\rho_{\Lambda}^*)+|\Lambda|S_{|\Lambda|}(\rho^*_{\Lambda})\right\}\left[\sqrt{2\pi D^{\alpha,-}_{\Lambda,\zero}(\rho^*_{\Lambda})|\Lambda|}+Ce^{-|\Lambda|^{2\alpha-1}}\right]
	\\
	&\le\exp\left\{|\Lambda|S_{|\Lambda|}(\rho^*_{\Lambda})\right\}\left[\sqrt{2\pi D^{\alpha,+}_{\Lambda,\zero}(\rho^*_{\Lambda})|\Lambda|}\left(1+E_{|\Lambda|}(\alpha,v,\rho_{\Lambda}^*)\right)\right]
	\end{split}
	\label{Sopra}
	\end{equation}
	with $C\in\R^+$ and where in the last but one inequality we used equation (iii) of Theorem 1.1 in \cite{bringmann2017asymptotic} and the fact that  $E_{|\Lambda|}(\alpha,v,\rho_{\Lambda}^*)=O(|\Lambda|^{-[m(\alpha)(1-\alpha)-1]})$. In the same way we obtain 
	\begin{equation}
	\begin{split}
	\sum_{N\in I_{\alpha,v}}\frac{e^{\beta\mu_0 N}\canBC(N)}{e^{\beta\mu_0 N^*}\canBC(N^*)}\ge e^{|\Lambda|S_{|\Lambda|}(\rho^*_{\Lambda})}\left[\sqrt{2\pi D^{\alpha,+}_{\Lambda,\zero}(\rho^*_{\Lambda})|\Lambda|}\left(1-E_{|\Lambda|}(\alpha,v,\rho_{\Lambda}^*)\right)\right]
	\end{split}
	\label{Sotto}
	\end{equation}
	which conclude the proof if $\alpha>(2d-1)/2d$. 
	
	Otherwise for $N\in I^{\alpha,v}_{\delta,v'}$, thanks to Lemma \ref{COR1} we have
	\begin{eqnarray}
	\sum_{N\in I^{\alpha,v}_{\delta,v'}}\frac{e^{\beta\mu_0 N}\canBC(N)}{e^{\beta\mu_0 N^*}\canBC(N^*)} & \s & e^{E_{|\Lambda|}(\delta,v,'\rho^*_{\Lambda})+|\Lambda|S_{|\Lambda|}(\rho^*_{\Lambda})}\sum_{N\in I^{\alpha,v}_{\delta,v'}}e^{-c \frac{(N-N^*)^2}{|\Lambda|}}\nonumber\\
	& \s & \exp\left\{-c v^2|\Lambda|^{2\alpha-1}\right\}
	\label{CodeIntermedie}
	\end{eqnarray}
	with $c\in\R^+$ independents on $|\Lambda|$.
\end{proof}

\begin{remark}	\label{RemPre}
	As we wrote in \eqref{lim} it is easy to see that $N^*/|\Lambda|,\;\bN/|\Lambda$ have as a limit $\rho_0$ when $|\Lambda|\rightarrow\infty$.
	Indeed from \eqref{GrandCan0}, \eqref{FreeE} and \eqref{FVP} we have
	\begin{equation}
	\beta p_{\Lambda,\beta,\zero}(\mu_0)
	=\frac{1}{|\Lambda|}\log\sum_{N\ge0}e^{\beta\mu_0 N}\canBC(N)
	\ge\frac{1}{|\Lambda|}\log\left[e^{\beta\mu_0 N^*}\canBC(N)\right]
	\end{equation} 
	On the other hand, let $N_{max}$ be the maximal number of particles allowed by the system, i.e., the first $N\in\mathbb{N}$ such that
	\begin{equation}
	e^{\beta\mu_0 N}\canBC(N)\le\frac{[e^{\beta(\mu_0+B)}|\Lambda|]^N}{N!}\le\left[\frac{\exp\left\{\beta(\mu_0+B)+\left(1-\frac{N}{12N+1}\right)\right\}|\Lambda|}{N}\right]^N<1.
	\label{N-max}
	\end{equation}
	Then we get
	\begin{eqnarray*}
		\beta p_{\Lambda,\beta,\zero}(\mu_0) & \leq &
		\frac{1}{|\Lambda|}\log\left[\left(e^{\beta\mu_0 N^*}\canBC(N^*)\right)\sum_{N=0}^{N^{max}}1+C\frac{[e^{\beta(\mu_0+B)}|\Lambda|]^N_{max}}{N_{max}!}\right]\nonumber
		\\
		& \leq & 
		\beta\mu_0\rho^*_{\Lambda}-\beta f_{\Lambda,\beta,\zero}(N^*)
		+\frac{1}{|\Lambda|}\log\left[\sum_{N=0}^{N_{max}}1+C\frac{[e^{\beta(\mu_0+B)}|\Lambda|]^N_{max}}{N_{max}!}\right],
	\end{eqnarray*}
	for some $C\in\R^+$.
\end{remark}

Now we can study the relation between $\bN$ and $N^*$. 
\begin{lemma}
	Let $\bN$ as in \eqref{MeanValue}  and $N^*$  which satisfies \eqref{N*} for $\mu=\mu_0$, such that condition $(\star)$ holds for both of them.
	
	We have:
	\begin{equation}
	\bN>N^*
	\label{1}
	\end{equation}	
	and
	\begin{equation}
	\bN-N^*\le C
	\label{2}
	\end{equation}
	for some $C>0$ which does not depend on $\Lambda$.
	\label{LemmaBN}
\end{lemma}	

\begin{proof}
	From \eqref{MeanValue}, adding and subtracting $N^*/|\Lambda|$ and multiplying and dividing by  $e^{\beta\mu_0N^*}\canBC(N^*)$,  we have:
	\begin{eqnarray}
	\bRL&=&\frac{\sum_{N\ge0}[(N\pm N^*)/|\Lambda|] e^{\beta\mu_0 N}\canBC(N)}{\GcanBC(\mu_0)}\nonumber
	\\
	&=&\frac{N^*}{|\Lambda|}+\frac{\sum_{N\ge0}[(N- N^*)/|\Lambda|] e^{\beta\mu_0 N}\canBC(N)}{\GcanBC(\mu_0)}\nonumber
	\\
	&=&\frac{N^*}{|\Lambda|}+\left[\sum_{N\ge0}\left(\frac{N- N^*}{|\Lambda|}\right) \frac{e^{\beta\mu_0N}\canBC(N)}{e^{\beta\mu_0N^*}\canBC(N^*)}\right]K(\mu_0,N^*)
	\label{FONDAMENTALE}
	\end{eqnarray}	
	which implies immediately \eqref{1}.
	
	The proof of \eqref{2} follows the strategy of the one of Lemma \ref{Lemma 3} with $\alpha=1/2$, such that we will use the sets defined in \eqref{set1}, \eqref{set2} and \eqref{set3}.
	
	Then from an equivalent study to the one done in \eqref{GrandiCode1}-\eqref{CodeGrandiForma1} we obtain that there exists $c\in\R^+$ such that 
	\begin{equation}
	\sum_{N\in I^{c}_{\delta,v'}}\left(\frac{N-N^*}{|\Lambda|}\right)e^{\beta\mu_0 (N-N^*)}\frac{\canBC(N)}{\canBC(N^*)}\;\s\;e^{-c|\partial\Lambda|}
	\end{equation}
	because for all $N\in I^{c}_{\delta,v'}$ we have
	\begin{equation}
	\frac{N-N^*}{|\Lambda|}\in\mathbb{Z}\cap\left\{\left[-\rho^*_{\Lambda},\frac{-v'}{|\Lambda|^{1-\delta}}\right]\cup\left[\frac{v'}{|\Lambda|^{1-\delta}},C(\beta,\mu_0,B)\right]\right\}
	\end{equation}
	where $C(\beta,\mu_0,B)$ is a constant which does not depend on $|\Lambda|$ and which can be derived from \eqref{N-max}.   
	
	For $N\in I^{1/2,v}_{\delta,v'}$ choosing $\delta$ such that $1-\delta=m(\delta)(1-\delta)-1$, i.e. $m(\delta)=1+1/(1-\delta)$, we get 
	\begin{eqnarray}
	&&\sum_{N\in I^{1/2,v}_{\delta,v'}}\left(\frac{N-N^*}{|\Lambda|}\right)e^{\beta\mu_0 (N-N^*)}\frac{\canBC(N)}{\canBC(N^*)}\s e^{|\Lambda|S_{|\Lambda|}(\rho^*_{\Lambda})}\times\nonumber
	\\
	&&\;\;\times\;\sum_{N\in I^{1/2,v}_{\delta,v'}}\left(\frac{N-N^*}{|\Lambda|}\right)e^{-\frac{(N-N^*)^2}{c|\Lambda|}+c_1\frac{N-N^*}{|\Lambda|}}\s e^{|\Lambda|S_{|\Lambda|}(\rho^*_{\Lambda})}\times\nonumber
	\\
	&&\times\sum_{N\in I^{1/2,v}_{\delta,v'}}\left(\frac{N-N^*}{|\Lambda|}\right)e^{-\frac{(N-N^*)^2}{c|\Lambda|}}\left(1+c_1\frac{N-N^*}{|\Lambda|}\right)\s e^{|\Lambda|S_{|\Lambda|}(\rho^*_{\Lambda})}\times\nonumber
	\\
	&&\;\;\times\sum_{N\in I^{1/2,v}_{\delta,v'}}\left(\frac{N-N^*}{|\Lambda|}\right)^2\;e^{-c\frac{(N-N^*)^2}{|\Lambda|}}\s \frac{e^{|\Lambda|S_{|\Lambda|}(\rho^*_{\Lambda})}}{|\Lambda|}
	\end{eqnarray}
	where $c,\;c_1\in\R^+$ independent on $|\Lambda|$ and because for all $N\in I^{1/2,v}_{\delta,v'}$ we have
	\begin{equation}
	\frac{(v')^2e^{-\frac{(v')^2|\Lambda|^{2\delta-1}}{c}}}{|\Lambda|^{2(1-\delta)}}\le\left(\frac{N-N^*}{|\Lambda|}\right)^2e^{-\frac{(N-N^*)^2}{c|\Lambda|}}\le\frac{v^2e^{-\frac{v^2}{c}}}{|\Lambda|}
	\end{equation}
	and  $[(N-N^*)/|\Lambda|]e^{-\frac{(N-N^*)^2}{c|\Lambda|}}$ is an odd function.
	
	Finally, using an estimate similar to the previous one we obtain 
	\begin{eqnarray}
	&&\sum_{N\in I_{1/2}}\left(\frac{N- N^*}{|\Lambda|}\right) e^{\beta\mu_0 (N-N^*)}\frac{\canBC(N)}{\canBC(N^*)}\s e^{|\Lambda|S_{|\Lambda|}(\rho^*_{\Lambda})}\nonumber
	\\
	&&\;\times\sum_{N\in I_{1/2}}\left(\frac{N-N^*}{|\Lambda|}\right)^2e^{-\frac{(N-N^*)^2}{2|\Lambda|D_{\Lambda,\beta,\e}}(\rho^*_{\Lambda})}\s e^{|\Lambda|S_{|\Lambda|}(\rho^*_{\Lambda})}\frac{\sqrt{|\Lambda|}}{|\Lambda|}.
	\end{eqnarray}
	
	The conclusions  follow from the fact that, thanks to Lemma \ref{Lemma 3} we have
	\begin{equation}
	K(\mu_0,N^*)\s\; \exp\left\{-|\Lambda|S_{|\Lambda|}(\rho^*_{\Lambda})\right\}\left(\sqrt{|\Lambda|}\right)^{-1}+\frac{1}{|\Lambda|}.
	\end{equation}
\end{proof}

\begin{remark}  	
	Note that \eqref{1} implies $A_{\bN}\subseteq A_{N^*}$ and then 	
	\begin{equation}
	\probBC(A_{N^*})\ge\probBC(A_{\bN}).
	\end{equation}
	
	On the other hand if we consider periodic boundary condition, thanks to the fact that from \cite{pulvirenti2015finite} and Lemma \ref{Lemma 1} we have
	\begin{equation}
	\left|\mu_0-\mathcal{F}'_{\Lambda,\beta,per}(\bRL)\right|\;\s\; \frac{1}{|\Lambda|},
	\end{equation}
	we can choose both $\bN$ and $N^*$ as center of deviations which implies  
	\begin{equation}
	\mathbb{P}^{per}_{\Lambda,\mu_0}(A_{\bar N})\sim\mathbb{P}^{per}_{\Lambda,\mu_0}(A_{N^*}).
	\end{equation}
	Here, $\mathbb{P}^{per}_{\Lambda,\mu_0}$ is the grand-canonical probability measure with periodic boundary conditions defined similarly to \eqref{GcProbE} where, instead of $H_{\Lambda}^{\zero}(\mathbf{q})$ (given by \eqref{Hamiltonian0}) we consider $H^{per}_{\Lambda}(\mathbf{q})$ defined using a proper \textquotedblleft periodic\textquotedblright stable and regular pair potential $V^{per}(x_i-x_j)$ (see for example definition (3) in \cite{pulvirenti2012cluster}).
	
\end{remark}

\appendix

\section{Canonical cluster expansion}
\label{appendice1}

Here we give some details about the cluster expansion needed for the proof of \eqref{SigmaINF}.
We follow the results from \cite{pulvirenti2012cluster} to which we refer for a more detailed description.

An \textit{abstract polymer model} $(\mathcal{V},\mathbb{G}_{\mathcal{V}},\omega)$ consists of (i) a set of polymers $\mathcal{V}:=\{V_1,...,V_{|\mathcal{V}|}\}$, (ii) a binary symmetric relation $\sim$ of compatibility on $\mathcal{V}\times\mathcal{V}$ such that for all $i\ne j$, $V_i\sim V_j$ if and only if $V_i\cap V_j=\emptyset$, (iii) a graph $\mathbb{G}_{\mathcal{V}}\equiv(V(\mathbb{G}_{\mathcal{V}}),E(\mathbb{G}_{\mathcal{V}}))$ such that the vertex set $V(\mathbb{G}_{\mathcal{V}})=\mathcal{V}$ and an edge $\{i,j\}\in E(\mathbb{G}_{\mathcal{V}})$ if and only if $V_i\not\sim V_j$ and (iv) a weight function $\omega:\mathcal{V}\rightarrow\mathbb{C}$. Then defining for all $N\in\mathbb{N}$
\begin{equation}
\mathcal{V}\equiv\mathcal{V}(N):=\{V\;:\;V\subset\{1,...,N\},\;|V|\ge2\}
\end{equation} 
and 
\begin{equation}
\omega\equiv\omega_{\Lambda}(V):=\sum_{g\in\mathcal{C}_V}\int_{\Lambda^{|g|}}\prod_{i\in V(g)}\frac{dq_i}{|\Lambda|}\prod_{\{i,j\}\in E(g)}f_{i,j}   
\end{equation}
with $f_{i,j}:=e^{-\beta V(q_i,q_j)}-1$, we have that thanks to \cite{pulvirenti2012cluster} 
\begin{equation}
\frac{N}{|\Lambda|}\sum_{n\ge1}\frac{1}{n+1}P_{N,|\Lambda|}(n)B_{\Lambda,\beta}(n)=\frac{1}{|\Lambda|}\sum_{I\in\mathcal{I}}c_I\omega_{\Lambda}^I,
\label{Polymer}
\end{equation} 
where
\begin{equation}
c_I=\frac{1}{I!}\sum_{G\subset\mathcal{G}_I}(-1)^{|E(G)|}=\frac{1}{I!}\frac{\partial^{\sum_{V}I(V)}\log Z_{\mathcal{V}(N),\omega_{\Lambda}}}{\partial^{I(V_1)}\omega_{\Lambda}(V_1)\cdot\cdot\cdot\partial^{I(V_n)}\omega_{\Lambda}(V_n)}\bigg|_{\omega_{\Lambda}(V)=0}
\label{C_I}
\end{equation}
with
\begin{equation}
Z_{\mathcal{V}(N),\omega_{\Lambda}}:=\sum_{\{V_1,...,V_n\}_{\sim}}\prod_{i=1}^n\omega_{\Lambda}(V_i)=\int_{\Lambda^N}\prod_{i=1}^N\frac{dq_i}{|\Lambda|}\;e^{-\beta H^{\zero}_{\Lambda}(\mathbf{q})}.
\end{equation}
The second sum in \eqref{Polymer} is over the set $\mathcal{I}$ of multi-indices $I:\mathcal{V}(N)\rightarrow\{0,1,...\}$, $\omega_{\Lambda}^{I}=\prod_{V}\omega_{\Lambda}(V)^{I(V)}$, and, denoting $\mathrm{supp}I:=\{V\in\mathcal{V}(N)\;:\; I(V)>0\}$, $\mathcal{G}_{I}$ is the graph with $\sum_{V\in\mathrm{supp}I}I(V)$ vertices induced from $\mathcal{G}_{\mathrm{supp}I}\subset\mathbb{G}_{\mathcal{V}(N)}$ by replacing each vertex $V$ by the complete graph on $I(V)$ vertices. 

Furthermore, the sum in \eqref{C_I} is over all connected graphs $G$ of $\mathcal{G}_I$ spanning the whole set of vertices of $\mathcal{G}_I$ and $I!=\prod_{V\in\mathrm{supp}I}I(V)!$.

Denoting now with $[N]\equiv\{1,..,N\}$ and $A(I):=\bigcup_{V\in\mathrm{supp}I}V\subset[N]$, from \eqref{Polymer} we have that
\begin{equation}
B_{\beta,\Lambda}(n)=\frac{|\Lambda|^n}{n!}\sum_{I\;:\;A(I)=[n+1]}c_I\omega_{\Lambda}^I.
\label{B}
\end{equation}
We want to recall that $B_{\Lambda,\beta}(n)$ is the {\it finite volume} version of the {\it irreducible Mayer's coefficient} $\beta_n$ in the sense that 
\begin{equation}
\lim_{\Lambda\rightarrow\R^d}B_{\Lambda,\beta}(n)=\beta_n,
\end{equation}
where
\begin{equation}
\beta_n:=\frac{1}{n!}\sum_{\substack{g\in\mathcal{B}_{n+1}\\V(g)\ni1}}\int_{(\R^d)^n}\prod_{\{i,j\}\in E(g)}(e^{-\beta V(x_i-x_j)}-1)dx_2\cdot\cdot\cdot dx_n,\;\;x_1\equiv0.
\label{Mayers}
\end{equation} 
Here $\mathcal{B}_{n+1}$ is the set of the 2-connected graphs and, for all $g\in\mathcal{B}_{n+1}$, we denote with  $V(g)$ the set of its vertices. 

Using this formalization, thanks to \cite{pulvirenti2015finite} and also assuming in order to simplify the calculation that

\textbf{Assumption 3:} $V:\R^d\times\R^d\rightarrow\R\cup\{\infty\}$ has compact support $R<l$, i.e. 
\begin{equation}
V(x_i-x_j)=0\;\;\;\mathrm{if}\;\;\;|x_i-x_j|>R,
\end{equation}
for all $x_i,x_j\in\R^d$, we have:
\begin{lemma}
	Let $\rho^*_{\Lambda}$ as in \eqref{mu_0} such that condition $(\star)$ holds and $\rho_0$ as in	\eqref{lim}.
	It results:
	\begin{equation}
	\beta\left|f^{(m)}_{\beta}(\rho_0)-\F^{(m)}(\rho^*_{\Lambda})\right|\s\frac{|\partial\Lambda|}{|\Lambda|},
	\label{aPrioriDeriv}
	\end{equation}
	for all $m\ge0$.
	\label{LemmaA1}
\end{lemma} 
\begin{proof}
	Let us consider for first the case $m\ge1$. We define 
	\begin{equation}
	\omega^{(m)}_{\Lambda}(V):=2\rho_{\Lambda}^{1-m}{{|V|+1}\choose m}\sum_{g\in\mathcal{C}_{V}}\int_{|\Lambda|^{|V|}}\prod_{i=1}^{|V|}\frac{dq_i}{|\Lambda|}\prod_{\{i,j\}\in E(g)} f_{i,j} \prod_{i=1}^{|V|}F_{q_i}(\epsilon)
	\label{Omega-m}
	\end{equation}
	where
	\begin{equation}
	F_{q}(\epsilon):=(1-\epsilon)\mathbf{1}_{\{d(q,\Lambda^c)< R|V|\}}+\epsilon\;\mathbf{1}_{\{d(q,\Lambda^c)\ge R|V|\}},
	\end{equation}

	Then, calling $n=|V|$ and using the same estimates of (4.18)-(4.20) of \cite{pulvirenti2015finite} we obtain
	\begin{equation}
	|\omega_{\Lambda}^{(m)}(V)|\le \rho_{\Lambda}^{1-m}\frac{2}{m!}\frac{e^{2\beta B}(n+1)^m}{n}\left[e^{2\beta B}C(\beta,R)\rho_{\Lambda}\right]^{n-1}
	\label{NuoviClusters1}
	\end{equation}
	far from the boundary ($\epsilon=1$), and
	\begin{equation}
	|\omega_{\Lambda}^{(m)}(V)|\le \rho_{\Lambda}^{1-m}\frac{2}{m!}\frac{e^{2\beta B}(n+1)^m}{n}\left[e^{2\beta B}C(\beta,R)\rho_{\Lambda}\right]^{n-1}\frac{dR}{l}
	\label{NuoviClusters2}
	\end{equation}
	near the boundary ($\epsilon=0$).
	Noting that $|\omega^{(m)}_{\Lambda}(V)|$ is an upper bound for the $n^{th}$ term of $\f_{\Lambda,\beta,\e}^{int,(m)}$ (see \eqref{CovClusrDeriv} and \eqref{Convergence}) and revisiting sections 4, 5 and 6 of \cite{pulvirenti2015finite} we have
	\begin{equation}
	\beta\bigg|f_{\beta}^{(m)}(\rho^*_{\Lambda})-\F^{(m)}(\rho^*_{\Lambda})\bigg|\s\frac{|\partial\Lambda|}{|\Lambda|}.
	\end{equation}   
	The conclusion follows (also when $m=0$) from the fact that by construction $|\rho^*_{\Lambda}-\rho_{0}|\s |\partial\Lambda|/|\Lambda|$ (see also \cite{de1986asymptotic}) and from the exponential decay of $F_{\Lambda,\beta,N}(n)$ given in \eqref{absCan}. Indeed from \eqref{Infinit}, \eqref{CanCE1} and \eqref{Cluster-Coefficient}-\eqref{Mayers} we have
	\begin{equation}
	\beta f_{\beta}^{(m)}(\rho_0)=\frac{d^m}{d\rho^m}\rho(\log\rho-1)\bigg|_{\rho=\rho_0}+\sum_{n\ge1}{n+1\choose m}\rho_0^{n+1-m}\frac{\beta_n}{n+1}.
	\label{fbeta}
	\end{equation}
	The first term gives
	
	\begin{equation}
	\frac{d^m}{d\rho^m}\rho(\log\rho-1)\bigg|_{\rho=\rho_0}=(-1)^m\frac{1}{\rho_0^{m-1}}=(-1)^m\left(\frac{1}{(\rho^*_{\Lambda})^{m-1}}+\frac{\sum_{k=1}^{m-1}{m-1\choose k}(\rho^*_{\Lambda})^{m-k}(\rho_0-\rho^*_{\Lambda})^k}{(\rho_0\rho^*_{\Lambda})^{m-1}}\right)
	\end{equation}	
	for all $m\ge2$. In the above formula, the first term cancels with the corresponding in $f_{\beta}^{(m)}(\rho^*_{\Lambda})$
	while the second is bounded by $|\partial\Lambda|/|\Lambda|$. 
	For the second term in \eqref{fbeta} we have:
	\begin{equation}\begin{split}\label{A.16}
	&\sum_{n\ge1}{n+1\choose m}\rho_0^{n+1-m}\frac{\beta_n}{n+1}=\sum_{n\ge1}{n+1\choose m}(\rho_0\pm\rho^*_{\Lambda})^{n+1-m}\frac{\beta_n}{n+1}
	\\
	&\;=\sum_{n\ge1}{n+1\choose m}\left[\sum_{k=0}^{n+1-m}{n+1-m\choose k}(\rho_0-\rho^*_{\Lambda})^{n+1-m-k}(\rho^*_{\Lambda})^k\right]\frac{\beta_n}{n+1}
	\\
	&\;=\sum_{n\ge1}{n+1\choose m}(\rho^*_{\Lambda})^{n+1-m}\frac{\beta_n}{n+1}+\sum_{n\ge1}{n+1\choose m}\sum_{k=0}^{n
		-m}\bigg[{n+1-m\choose k}\times
	\\
	&\;\;\;\;\;\;\;\;\;\;\;\;\;\;\;\;\;\;\;\;\;\;\;\;\;\;\;\;\;\;\;\;\;\;\;\;\;\;\;\;\;\;\;\;\;\;\;\;\;\;\;\;\;\;\;\;\times(\rho_0-\rho^*_{\Lambda})^{n+1-m-k}(\rho^*_{\Lambda})^k\bigg]\frac{\beta_n}{n+1},
	\end{split}
	\end{equation} 
	where
	\begin{equation}\begin{split}\label{A.17}
	&\sum_{k=0}^{n-m}\left[{n+1-m\choose k}(\rho_0-\rho^*_{\Lambda})^{n+1-m-k}(\rho^*_{\Lambda})^k\right]\frac{\beta_n}{n+1}
	\\
	&	\le\frac{\rho^*_{\Lambda}}{2^m}\frac{|\partial\Lambda|}{|\Lambda|}\sum_{n\ge1}(n+1)^m\left(\frac{2}{\rho^*_{\Lambda}}\right)^n\left|\frac{\rho^*_{\Lambda}}{n+1}\beta_n\right|\le\frac{\rho^*_{\Lambda}}{2^m}\frac{|\partial\Lambda|}{|\Lambda|}\sum_{n\ge1}\left(\frac{2}{\rho^*_{\Lambda} e^c}\right)^n\s\frac{|\partial\Lambda|}{|\Lambda|}.
	\end{split}
	\end{equation}
	which concludes the proof.
\end{proof}

\begin{remark}
	From \eqref{lim}, i.e. the fact that $|\rho_0-\bRL|\s|\partial\Lambda|/|\Lambda|$ with $\bRL$ given by \eqref{MeanValue}, the previous Lemma is also valid if we consider $\bRL$ instead of $\rho^*_{\Lambda}$.
\end{remark}

\section{Stirling's approximation}
\label{appendice2}
We recall Stirling's formula: for $N\in\mathbb{N}$ large enough
\begin{equation}
\sqrt{2\pi N}\left(\frac{N}{e}\right)^N\le N!\le e^{1/12N}\sqrt{2\pi N}\left(\frac{N}{e}\right)^N.
\label{primo}
\end{equation}

Using \eqref{FreeE}, \eqref{CanCE1} and \eqref{FreeERL1}, for $\rho_{\Lambda}=N/|\Lambda|\in(0,1)$ we have:
\begin{eqnarray}
\beta[f_{\Lambda,\beta,\zero}(N)-\F(\rho_{\Lambda})] & = & -\frac{1}{|\Lambda|}\log\frac{|\Lambda|^{\rho_{\Lambda}|\Lambda|}}{(\rho_{\Lambda}|\Lambda|)!}-\rho_{\Lambda}(\log\rho_{\Lambda}-1)\nonumber
\\
&=&
\frac{1}{|\Lambda|}\log\left[(\rho_{\Lambda}|\Lambda|)!\left(\frac{e}{\rho_{\Lambda}|\Lambda|}\right)^{\rho_{\Lambda}|\Lambda|}\right]\nonumber\\
& =: & \rho_{\Lambda}B_{|\Lambda|}(\rho_{\Lambda}|\Lambda|) =:
S_{|\Lambda|}(\rho_{\Lambda}).
\label{StirlingEq}
\end{eqnarray} 

Thus, from \eqref{primo}, we get
\begin{equation}
\frac{\log\sqrt{2\pi\rho_{\Lambda}|\Lambda|}}{|\Lambda|}\le S_{|\Lambda|}(\rho_{\Lambda})\le\frac{\log\sqrt{2\pi\rho_{\Lambda}|\Lambda|}}{|\Lambda|}+\frac{1}{12\rho_{\Lambda}|\Lambda|^2}.
\label{StirlingBound}
\end{equation}

We can generalize the first quantity defined in \eqref{StirlingEq} as follows
\begin{equation}
B(x)=\frac{1}{x}\log\left[\Gamma(x+1)\left(\frac{x}{e}\right)^{-x}\right]
\end{equation}
for all $x\in\R^+$ and where $\Gamma(\cdot)$ is the gamma function $\Gamma(x):=\int_{0}^{\infty}t^{x-1}e^{-t}dt$ with the following properties \cite{jameson2015simple}:
\begin{equation}
\Gamma(N+1)=N!
\end{equation}
for all $N\in\mathbb{N}$,
\begin{equation}
\sqrt{\frac{2\pi}{x}}\left(\frac{x}{e}\right)^{x}\le \Gamma(x)\le \sqrt{\frac{2\pi}{x}}\left(\frac{x}{e}\right)^{x}e^{\frac{1}{12x}}
\end{equation}
for all $x\in\R^+$.  Moreover, 
\begin{equation}
\psi(x):=\frac{d}{dx}\log(\Gamma(x))=\frac{\Gamma'(x)}{\Gamma(x)}
=\log x-\frac{1}{12x}-p(x),
\end{equation}
for all $x\in\R^+$ and where $0\le p(x)\le1/12x^2$.
Then we have:
\begin{eqnarray}
\frac{d}{dx}B(x) & = & \frac{\psi(x+1)-\log x}{x}-\frac{B(x)}{x}\nonumber\\
&=&\frac{1}{x}\left[\log\left(1+\frac{1}{x}\right)+\frac{1}{12(x+1)}+p(x+1)\right] -\frac{B(x)}{x},
\end{eqnarray}
so that, denoting with $B'(\rho_{\Lambda}|\Lambda|)=d/dx B(x)|_{x=\rho_{\Lambda}|\Lambda|}$ 
we have 
\begin{eqnarray}
S'_{|\Lambda|}(\rho_{\Lambda}) & = & B(\rho_{\Lambda}|\Lambda|)+\rho_{\Lambda}|\Lambda|B'(\rho_{\Lambda}|\Lambda|)=\frac{13\rho_{\Lambda}|\Lambda|+12}{12\rho_{\Lambda}|\Lambda|(\rho_{\Lambda}|\Lambda|+1)}\nonumber
\\
&+&\sum_{n\ge2}\frac{(-1)^{n+1}}{n}\left(\frac{1}{\rho_{\Lambda}|\Lambda|}\right)^n+p(\rho_{\Lambda}|\Lambda|+1)\;\s\;\frac{1}{|\Lambda|}.
\label{Stirling1}
\end{eqnarray}

\section*{Acknowledgment} I would like to thank Dimitrios Tsagkarogiannis 
for his patient supervision during the preparation of this paper. I would also like
to thank Errico Presutti and Sabine Jansen for their helpful comments. 

\bibliographystyle{plain}
\bibliography{bibliografia1}
\nocite{*}

\end{document}